\documentclass[preprint,twocolumn,5p]{elsarticle}

\usepackage[utf8]{inputenc}
\usepackage{amsmath}
\usepackage{amsfonts}
\usepackage{amssymb}
\usepackage{amsthm}
\usepackage{booktabs}
\usepackage{graphicx}
\usepackage{hyperref}
\usepackage{import}
\usepackage{mathtools}
\usepackage{microtype}
\usepackage{multirow}
\usepackage{pdfpages}
\usepackage{placeins}
\usepackage{stix2}
\usepackage{tikz}
\usepackage{transparent}
\usepackage{xcolor}

\usepackage{cleveref}

\newtheorem{theorem}{Theorem}

\title{Exact expressions for the unresolved stress in a finite-volume based large-eddy simulation}

\date{\today}

\begin{document}

\begin{frontmatter}

\author[cwi,eindhoven]{Syver Døving Agdestein\corref{cor1}}
\ead{sda@cwi.nl}
\author[groningen]{Roel Verstappen}
\ead{r.w.c.p.verstappen@rug.nl}
\author[cwi,eindhoven]{Benjamin Sanderse}
\ead{b.sanderse@cwi.nl}

\cortext[cor1]{Corresponding author}

\affiliation[cwi]{
    organization={Scientific Computing Group, Centrum Wiskunde \& Informatica},
    addressline={Science Park 123}, 
    city={Amsterdam},
    postcode={1098 XG}, 
    country={The Netherlands}
}
\affiliation[eindhoven]{
    organization={Centre for Analysis, Scientific Computing and Applications, Eindhoven University of Technology},
    addressline={PO Box 513}, 
    city={Eindhoven},
    postcode={5600 MB}, 
    country={The Netherlands}
}
\affiliation[groningen]{
    organization={
        Bernoulli Institute for
        Mathematics,
        Computer Science and
        Artificial Intelligence,
        University of Groningen
    },
    addressline={Nijenborgh 9}, 
    city={Groningen},
    postcode={9747 AG}, 
    country={The Netherlands}
}

\begin{abstract}
In this article we propose new discretization-informed expressions for the
residual stress tensor (RST) in a finite-volume based large-eddy simulation
(LES-FVM).
In addition to the classical RST $\overline{u u} - \bar{u} \bar{u}$
resulting from the non-commutation between filtering and the nonlinear stress,
our RST also contains contributions from the numerical flux,
discrete divergence, and pressure terms.
Unlike the classical RST, our proposed RST is non-symmetric and non-local.
The proposed form of the RST is important for generating appropriate reference
data for LES closure modeling.

Based on DNS results of the 1D Burgers and 3D incompressible Navier-Stokes
equations, we show that the discretization-induced parts of the RST
play an important role in the LES-FVM equation for common LES filter
widths.
When the discrete contribution is included, our RST expression gives zero
a-posteriori error in LES, while existing RST expressions give errors that
increase over time. For a Smagorinsky model, we show that the Smagorinsky
coefficient is higher when fitted to our new RST than when fitted to the
classical RST and gives improved results.

\end{abstract}

\begin{graphicalabstract}
\begin{center}
\includegraphics[width=\textwidth]{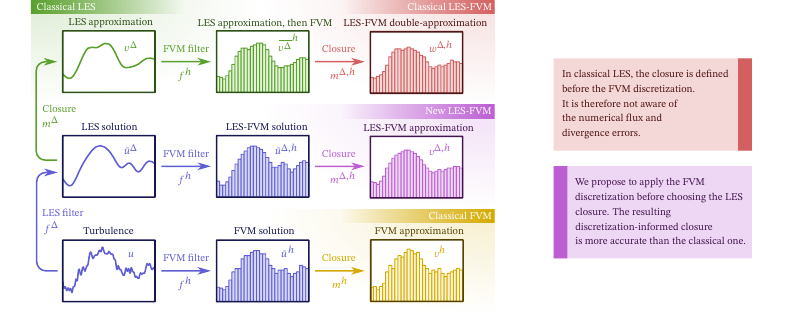}
\end{center}
\end{graphicalabstract}

\begin{highlights}

\item
We unite large-eddy simulation (LES) and the finite-volume method (FVM) into LES-FVM.

\item
We derive an exact expression for the residual stress tensor (RST) in LES-FVM.

\item
For incompressible flows, the RST is shown to be non-local and non-symmetric.

\item
In a DNS-aided LES, our RST gives zero a-posteriori error, unlike the classical RST.

\item
Accounting for the new RST improves the performance of a Smagorinsky model.

\end{highlights}

\begin{keyword}
commutator errors \sep
closure modeling \sep
data-consistency \sep
filtering \sep
finite volume method \sep
large-eddy simulation \sep
sub-filter stress \sep
turbulence
\end{keyword}

\end{frontmatter}

\section{Introduction}

Turbulent fluid flows can be modeled by the incompressible Navier-Stokes
equations, but they are in general computationally too expensive to solve using
direct numerical simulation (DNS). Large eddy simulation (LES) consists of
finding equations for the large-scale features of the flow, which are extracted
using a spatial filter. The LES equations can be solved using fewer numerical
computations.

The incompressible Navier-Stokes equations can be discretized using the finite
volume method (FVM). Like LES, the FVM considers filtered velocities, with the
filter being the average over a control volume. In LES, this filter can have
other definitions, for example a convolution with an arbitrary kernel function.
Unlike LES, the FVM equations are discrete by design.

The continuous LES equations include a continuous divergence of a flux function.
The discrete FVM equations contain an integral of a flux function over the
control volume boundary.
By defining this boundary integral as a \emph{discrete} divergence operator
applied to the given flux,
the FVM equations take the same form as the continuous LES equation.
This allows for treating LES and the FVM as the same problem, as suggested by
Verstappen~\cite{verstappenMergingFilteringModeling2025}.

Both the FVM and LES result in approximate equations for the filtered velocity
field $\bar{u}$ that are different from the exact filtered conservation law. The
mismatch between the model equations and the exact filtered conservation law can
be written as a residual term, sometimes called a commutator error.

Most works on the closure problem focus on
modeling the commutator between filtering and
nonlinearities~\cite{
popeTurbulentFlows2000,
berselliMathematicsLargeEddy2006,
sagautLargeEddySimulation2006}.
Assuming filtering commutes with
spatial and temporal differentiation, the residual term
takes the form of the divergence of a tensor,
$\nabla \cdot \tau(u)$,
which is a function of the resolved and unresolved velocity
fields $\bar{u}$ and $u$ through the residual stress tensor (RST)
$\tau_{i j}(u) \coloneq \overline{u_i u_j} - \bar{u}_i \bar{u}_j$.
The exact LES equations are approximated by replacing $\tau(u)$ with a closure
model $m(\bar{u})$. \emph{Functional} models try to match the dissipation
produced by $\tau(u)$,
while \emph{structural} models also try to model the structure of the RST
$\tau(u)$ itself~\cite{
sagautLargeEddySimulation2006,
luStructuralSubgridscaleModeling2016}.

Apart from the difficulty of correctly modeling the residual
$\nabla \cdot \tau(u)$,
another major challenge in LES is the appearance of additional
residual terms when the equations are
discretized~\cite{baeNumericalModelingError2022,
beckDiscretizationConsistentClosureSchemes2023}.
These terms are commonly referred to as
\emph{discretization errors}.
For common discretization schemes, they are controlled by convergence
bounds on the grid spacing $h$.

For typical LES scenarios, the filter width $\Delta$ is of the same order of
magnitude as the grid spacing $h$~\cite{
ghosalAnalysisNumericalErrors1996,
chowFurtherStudyNumerical2003}, sometimes with an exact equality
$\Delta = h$~\cite{deardorffMagnitudeSubgridScale1971,
normandDirectLargeeddySimulations1992,
vremanComparisonNumericalSchemes1996,
geurtsNumericallyInducedHighpass2005}.
For this reason the discretization errors can be of the same order of
magnitude as
$\nabla \cdot \tau(u)$~\cite{clarkEvaluationSubgridscaleModels1979},
which can strongly limit the usefulness of any closure modeling effort that only
accounts for $\nabla \cdot \tau(u)$. This was illustrated by Bae and
Lozano-Duran, who used DNS to assess the importance of the discretization
error by computing the residual $\nabla \cdot \tau(u)$
explicitly~\cite{baeNumericalModelingError2022}.
A recognition of this issue has
motivated the development of LES frameworks where the total residual from
both non-linearities and discretization are modeled~\cite{
geurtsNumericallyInducedHighpass2005,
dairayNumericalDissipationVs2017,
beckDiscretizationConsistentClosureSchemes2023}.
However, the exact form of this total residual term has not yet been well-defined.

We are interested in the exact discretization-informed expression for the
total residual because the final goal is to use such an expression as
target data for tuning the parameters of a closure model.
In previous work, we showed that using a discretization-informed residual
as target data for data-driven closure models gives stable models,
while using incorrect target data leads to
instabilities~\cite{agdesteinDiscretizeFirstFilter2025}.
In this work, we introduce a new mathematical formalism to derive exact
expressions for the residual appearing in the discretized LES equations.
Furthermore, we write the residual in the form of a discrete divergence of an
RST. We show how this RST differs between classical LES, the classical FVM, and
propose a new discrete LES-FVM framework that merges the two.

\begin{figure*}[t]
    \centering
    \def\svgwidth{\textwidth}
    \includegraphics[width=\textwidth]{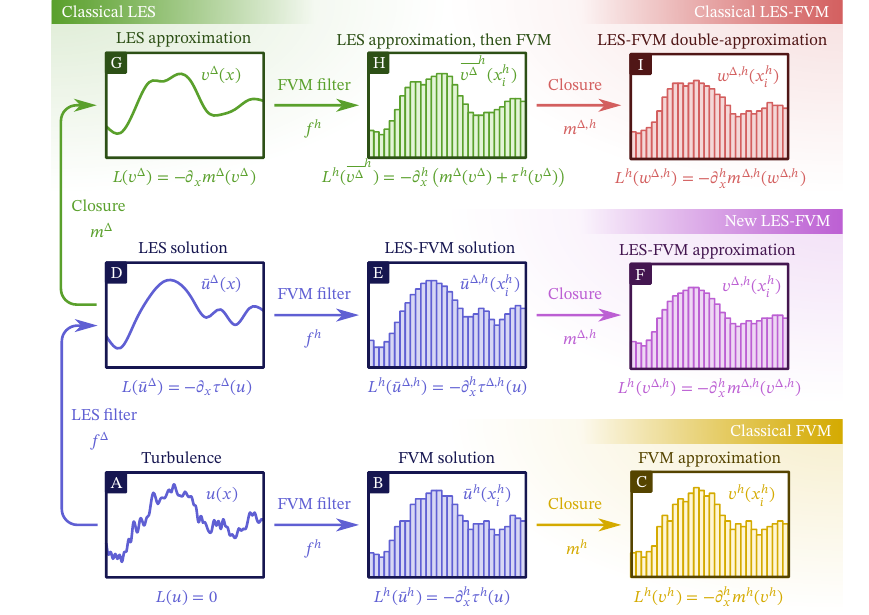}
    \caption{
        Three different routes to a closed system of discrete equations that
        simulate the large-scales of a turbulent flow $u$.
        A change of color indicates an approximation step,
        while a preservation of color indicates an exact operation.
        The grid points are $(x^h_i)_{i \in \mathbb{Z}}$.
        Classical LES-FVM relies on two separate
        approximation steps, whereas the FVM and our proposed LES-FVM
        framework only rely on one approximation step.
    }
    \label{fig:threepaths}
\end{figure*}

The main ideas in this article are visualized in \cref{fig:threepaths}.
The notation used in the figure will be explained in detail in the following sections.
All equations are derived from the original conservation law for the
turbulent solution $u$ in \cref{fig:threepaths}-A.
The approach we use is to first transform the equation for $u$ by applying LES
and FVM filters.
The equations are then rewritten into the desired form, and the residual terms 
are gathered in an unresolved stress.
The approximation steps (called ``closures'')
are only performed after the exact residual stress has been defined.
These approximation steps are indicated by a change of color in the figure.

In \cref{sec:conservation-law},
we first introduce our notation.
For 1D conservation laws, we derive classical LES and the
FVM in this notation, so that in \cref{sec:les-fvm}, we can merge classical LES
and the FVM into a unified discrete LES-FVM framework. We derive new exact
discretization-informed RST expressions for the given framework.
In \cref{sec:burgers},
the importance of the RST definition is tested for the 1D
Burgers equation. We show that using our discretization-informed RST as
a closure term gives a ``perfect'' closure model with zero a-posteriori error,
while the classical RST gives errors that accumulate over time. This difference
remains visible when the RSTs are used as target data in a Smagorinsky model.
In \cref{sec:navier-stokes}, we extend our LES-FVM framework to the 3D
incompressible Navier-Stokes equations.
Due to the incompressibility, the RST is shown to be non-local.
The RST is also shown to be non-symmetric.
In \cref{sec:navier-stokes-experiments}, we repeat the Burgers experiments for a
3D decaying turbulence simulation, confirming that our expressions are correct
for the full incompressible Navier-Stokes equations.

\section{Preliminaries: LES and the FVM} \label{sec:conservation-law}

Let $\Omega = [0, \ell]$ be a 1D domain with length $\ell > 0$.
Let $U \coloneq \{ u : \mathbb{R} \to \mathbb{R} \mid u(x) = u(x + \ell) \}$
be the space of periodic functions on $\Omega$.
Depending on the problem,
the space $U$ may need to be further restricted
to more regular spaces such as
$L^2(\Omega)$ or $H^1(\Omega)$~\cite{berselliMathematicsLargeEddy2006}.

Consider the generic infinitesimal 1D conservation law
\begin{equation} \label{eq:conservation-law}
    L(u) \coloneq \partial_t u + \partial_x r(u) = 0
\end{equation}
(see \cref{fig:threepaths}-A),
where $u(x, t)$ is the solution at a given point $x$ and time $t$,
$L \coloneq \partial_t + \partial_x r$ is the equation operator,
$\partial_t \coloneq \partial / \partial t$ and
$\partial_x \coloneq \partial / \partial x$ are partial derivatives,
and the flux $r : U \to U$ is a non-linear spatial operator.
We call the conservation law \emph{infinitesimal}
since the divergence $\partial_x$ and potentially the flux $r$ are
infinitesimal operators
(as opposed to discrete operators that can be computed on a grid).
For simplicity, we did not include a source term in \eqref{eq:conservation-law}.
Adding a source term is straightforward and does not change the
derivations in this article.

The conservation law \eqref{eq:conservation-law} is a continuous equation defined
by requiring that the field $L(u) \in U$ is zero everywhere.
We can evaluate $L(u)$ in a point $x \in \Omega$ and time $t \geq 0$
as $L(u)(x, t) \in \mathbb{R}$.
In the following, we omit the time $t$ and write
$L(u)(x) \in \mathbb{R}$ and $u(x) \in \mathbb{R}$ etc.

An example of a non-linear conservation law is the viscous Burgers equation.
The corresponding flux is defined as
\begin{align}
    r(u) & \coloneq \frac{1}{2} u u - \nu \partial_x u, \label{eq:burgers}
\end{align}
where $\nu > 0$ is a constant viscosity
(diffusion coefficient).

The PDE \eqref{eq:conservation-law} can be discretized and solved directly
by using the FVM.
Alternatively, the equation can first be filtered and
approximated with an LES closure.
To distinguish between these two approaches, we use the superscript
$(\cdot)^\Delta$ for quantities related to LES, and $(\cdot)^h$ for
quantities related to the FVM.

\subsection{Filtering and LES} \label{sec:les}

Equation \eqref{eq:conservation-law} describes all the scales of motion for
the given system.
Filtering \eqref{eq:conservation-law} with a convolutional LES filter
\begin{equation}
    f^\Delta : U \to U, \ u \mapsto \bar{u}^\Delta
\end{equation}
of filter-width $\Delta \geq 0$
gives the LES equation
\begin{equation} \label{eq:filtered-conservation-law}
    \overline{L(u)}^\Delta = 0,
\end{equation}
where $\bar{u}^\Delta \coloneq f^\Delta u$
is the LES solution that we intend to solve for.

The convolution $f^\Delta$ is defined through a kernel
$G^\Delta : \mathbb{R} \to \mathbb{R}$ by
\begin{equation} \label{eq:convolution}
    \bar{u}^\Delta(x) \coloneq \int_\mathbb{R} G^\Delta(x - y) u(y) \, \mathrm{d} y
\end{equation}
for all $u \in U$ and $x \in \Omega$.
Note that we integrate over $\mathbb{R}$ (not $\Omega$), since the filter
kernel $G^\Delta$ needs to be extended beyond the periodic boundary.
Some kernels (such as the Gaussian kernel) have infinite support.
In practice, such kernels are truncated, and one periodic extension
is sufficient.

It is common to decompose the LES equation into a resolved and
unresolved part as
\begin{equation} \label{eq:cl-non-conservative}
\begin{split}
    L(\bar{u}^\Delta)
    & = -\left( \overline{L(u)}^\Delta - L(\bar{u}^\Delta) \right) \\
    & = -\left(\overline{\partial_x r(u)}^\Delta - \partial_x r(\bar{u}^\Delta)\right),
\end{split}
\end{equation}
where the left-hand side only depends on the resolved scales $\bar{u}^\Delta$
and the right-hand side is a residual that still depends on the full
solution $u$.
This term is not yet the ``divergence of a flux'', meaning that
\eqref{eq:cl-non-conservative} is not expressed as a conservation law.
However, since $f^\Delta$ is a convolution,
it satisfies the \emph{filter-swap} commutation property
\begin{equation} \label{eq:continuous-commutation}
    f^\Delta \partial_x = \partial_x f^\Delta.
\end{equation}
For proof, see \cref{th:commutation-f-partial} in \ref{sec:proofs}.
This can also be written as
$\overline{\partial_x u}^\Delta = \partial_x \bar{u}^\Delta$ for all $u \in U$.
This commutation property can be used to
obtain a \emph{conservative} form
of the residual and rewrite \cref{eq:cl-non-conservative}
as a conservation law
\begin{equation} \label{eq:les}
    L(\bar{u}^\Delta) = -\partial_x \tau^\Delta(u)
\end{equation}
(see \cref{fig:threepaths}-D),
where the commutator between $L$ and $f^\Delta$ takes the form of the
divergence of a residual flux
\begin{equation}
    \boxed{\tau^\Delta(u) \coloneq \overline{r(u)}^\Delta - r(\bar{u}^\Delta)}.
\end{equation}
This flux is also known as the \emph{sub-filter} flux.
For the Burgers equation, we get the well-known expression
$\tau^\Delta(u) = (\overline{u u}^\Delta - \bar{u}^\Delta \bar{u}^\Delta) / 2$.
As $\tau^\Delta(u)$ is a flux,
it has dissipation properties that could potentially be replicated by a
closure model.
For example, a convective flux conserves energy, while a diffusive flux
dissipates energy. Since $\partial_x \tau^\Delta(u)$ is of conservative form, 
the filtered momentum is conserved,
which cannot be guaranteed otherwise.
We highlight these properties and how they are obtained here since they do not
automatically apply in the discrete case.

Equation \eqref{eq:les} is exact, but unclosed.
The next step in LES is to approximate the residual flux by a closure model
$m^\Delta(\bar{u}^\Delta) \approx \tau^\Delta(u)$.
\footnote{
    We use the symbol ``$\approx$'' solely to indicate the \emph{intent} of an
    approximation to guide the reader.
    The statement ``$a \approx b$'' provides no guarantee that $a$ is in any
    way similar or close to $b$.
    Instead, we take care to use different symbols for new quantities arising
    from an approximation step, for example $v^\Delta \neq \bar{u}^\Delta$.
}
Closure models are often chosen in the functional eddy-viscosity form:
\begin{equation} \label{eq:smagorinsky}
    m^\Delta(\bar{u}^\Delta) \coloneq
    -\nu^\Delta(\bar{u}^\Delta) \partial_x \bar{u}^\Delta,
\end{equation}
where $\nu^\Delta : U \to U$ is an eddy viscosity
chosen such that
the dissipation produced by $m^\Delta(\bar{u}^\Delta)$ is similar
to the one produced by $\tau^\Delta(u)$.
For the basic Smagorinsky
model~\cite{smagorinskyGeneralCirculationExperiments1963}, the viscosity is
\begin{equation}
    \nu^\Delta(\bar{u}^\Delta) \coloneq \theta^2 \Delta^2 | \partial_x \bar{u}^\Delta |,
\end{equation}
where $\theta > 0$ is a model parameter.

The approximate LES equation is
\begin{equation} \label{eq:cl-conventional-les}
    L(v^\Delta) = -\partial_x m^\Delta(v^\Delta)
\end{equation}
(see \cref{fig:threepaths}-G),
where $v^\Delta \approx \bar{u}^\Delta$ is the approximate LES solution.
It is in general different from $\bar{u}^\Delta$ since the closure $m^\Delta$
cannot be exact when information is lost in the filtering process.

To summarize, we used the following steps to derive a closed form 
for the approximate equation in LES:
\begin{enumerate}
    \item Apply the LES filter $f^\Delta$.
    \item Decompose the equation into a resolved part and a residual.
    \item Rewrite the residual in conservative form using the filter-swap property.
    \item Approximate the residual flux with a closure model.
\end{enumerate}
These steps are visualized in the route A$\to$D$\to$G in \cref{fig:threepaths}.

\Cref{eq:cl-conventional-les} still needs to be discretized.
The discretization we use is the FVM.
As we will see, the equations for the FVM can be derived in a similar way as for LES.

\subsection{The FVM through filter-swap}
\label{sec:fvm}

Let $h > 0$ denote the grid spacing of a uniform grid on $\Omega$.
The derivations in this section hold for all points $x \in \Omega$,
and we do not restrict any quantities to the grid points
until \cref{sec:burgers}.
For now, $h$ can be considered as a particular
length scale that is independent from the LES filter width $\Delta$ from
\cref{sec:les}.

\subsubsection{Numerical derivatives and fluxes}

For all $u \in U$ and $x \in \Omega$,
define the staggered central finite difference and
interpolation operators as
\begin{align}
    \label{eq:cl-finite-difference} 
    \partial^h_x u(x)
    & \coloneq \frac{1}{h} \left[
        u\left(x + \frac{h}{2}\right) -
        u\left(x - \frac{h}{2}\right)
    \right], \\
    \label{eq:cl-interpolator}
    \eta^h_x u(x)
    & \coloneq \frac{1}{2} \left[
        u\left(x - \frac{h}{2}\right) +
        u\left(x + \frac{h}{2}\right)
    \right].
\end{align}
These operators can be used to discretize the
conservation law \eqref{eq:conservation-law}.
If $r^h \approx r$ is a numerical flux, then
$\partial^h_x r^h \approx \partial_x r$ is an approximation of the flux term,
and $L^h \coloneq \partial_t + \partial^h_x r^h$ is a
(spatial) approximation of the conservation law.
For the Burgers equation
(defined by \cref{eq:burgers}),
we use the second-order accurate numerical flux
\begin{equation}
    \label{eq:burgers-discrete}
    r^h(u)
    \coloneq \frac{1}{2} \left(\eta^h_x u\right) \left(\eta^h_x u\right)
    - \nu \partial^h_x u.
\end{equation}

The operators $\partial^h_x : U \to U$, $\eta^h_x : U \to U$, and $r^h : U \to U$
are both continuous and discrete at the same time.
They are continuous in the sense that $\partial^h_x u(x)$ can be evaluated for all $x \in \Omega$.
They are discrete in the sense that $\partial^h_x u(x)$ only depends on $u(x \pm h / 2)$,
and for the numerical Burgers flux \eqref{eq:burgers-discrete},
$\partial^h_x r^h(u)(x)$ only depends on $u(x - h)$, $u(x)$, and $u(x + h)$.
We therefore say that these operators are \emph{grid-compatible} or
$h$-compatible.
See \ref{sec:grid-compatibility} for more details about
infinitesimal, discrete, continuous, and grid-compatible operators.

\subsubsection{The FVM grid filter}

The finite difference $\partial^h_x$ is closely related to the FVM grid filter
\begin{equation}
    f^h : U \to U, \ u \mapsto \bar{u}^h
\end{equation}
defined for all $x \in \Omega$ as
\begin{equation} \label{eq:cl-gridfilter}
    \bar{u}^h(x) \coloneq \frac{1}{h}
    \int_{x - h / 2}^{x + h / 2} u(y) \, \mathrm{d} y.
\end{equation}
This filter is sometimes called a ``top-hat filter'',
``Schumann's filter''~\cite{schumannSubgridScaleModel1975},
or ''volume-averaging filter'', since it averages $u$ over a control volume
$[x \pm h / 2]$. The FVM filter $f^h$ constitutes a particular choice of
convolutional LES filter $f^\Delta$, where the filter width $\Delta$ is equal
to the grid spacing $h$ used in the finite difference $\partial^h_x$. The
underlying top-hat kernel $G^h$ is
\begin{equation}
    G^h(x) \coloneq
    \begin{cases}
        \frac{1}{h} & \text{if } |x| \leq \frac{h}{2}, \\
        0 & \text{otherwise.}
    \end{cases}
\end{equation}
A filter width equal to the grid spacing is commonly the setting in implicit LES,
but we emphasize that here the FVM filter definition is known explicitly.

The infinitesimal operator $\partial_x$ is not grid-compatible.
By applying the FVM filter $f^h$, the derivative $\partial_x$ can be turned
into the grid-compatible finite difference $\partial^h_x$ through the
discrete filter-swap property
\begin{equation} \label{eq:cl-continuous-commutation}
    \partial^h_x = f^h \partial_x,
\end{equation}
which can also be written as
$\partial^h_x u = \overline{\partial_x u}^h$ for all $u \in U$.
The proof is given in \cref{th:commutation-fh}.
This property entails that the finite difference $\partial^h_x$
is equal to a filtered version of the exact derivative $\partial_x$,
as noted by Schumann and
others~\cite{schumannSubgridScaleModel1975,
rogalloNumericalSimulationTurbulent1984,
lundUseExplicitFilters2003}.
If we approximate an infinitesimal derivative $\partial_x$ by $\partial^h_x$, the
content of the derivative is implicitly filtered.
The FVM filter $f^h$ is therefore 
sometimes called a ``discretization-induced filter''
or an ``implicit filter'' induced by the choice of $\partial^h_x$.
We prefer to make this statement explicit by turning it around and
explicitly applying $f^h$ to a derivative $\partial_x$.
The finite difference $\partial^h_x$ is therefore a
``filter-induced discretization'' of the infinitesimal derivative $\partial_x$.

We see the property \eqref{eq:cl-continuous-commutation}
as an analogous version of the commutation property \eqref{eq:continuous-commutation}
for the finite difference operator $\partial^h_x$ and grid filter $f^h$.
The subtle, yet crucial, difference with \eqref{eq:continuous-commutation}
is that the left-hand side is no longer filtered and uses a discrete
differentiation operator.
An intuitive way to understand this property is through the use of a telescoping sum.
See \cref{fig:commutation} in \ref{sec:proofs} for a visualization
of how the inner terms in the integral defining $f^h$ cancel out in the
discrete filter-swap property.
Unlike the property \eqref{eq:continuous-commutation},
we have $\partial^h_x f^h \neq f^h \partial_x$
(see \cref{th:noncommutation-coarsegraining} for proof).

\subsubsection{FVM equations}

To obtain FVM equations, we use the same approach as we used for LES
in \cref{sec:les}:
\begin{enumerate}
    \item Apply the FVM filter $f^h$.
    \item Decompose the equation into a resolved part and a residual.
    \item Rewrite the residual in conservative form using the discrete
        filter-swap property.
    \item Approximate the residual flux with a closure model.
\end{enumerate}
These steps are visualized in the bottom route in \cref{fig:threepaths}.

Applying the FVM filter $f^h$ to \cref{eq:conservation-law} yields the FVM
equation
\begin{equation} \label{eq:cl-fh}
    \overline{L(u)}^h = 0,
\end{equation}
where $\bar{u}^h$ is the FVM solution that we intend to solve for.
While the classical LES formulation could be obtained by using the 
filter-swap property $f^h \partial_x = \partial_x f^h$,
our aim is to have an equation for $\bar{u}^h$ that involves the discrete
divergence $\partial^h_x$. This is achieved by applying the
discrete filter-swap property to \cref{eq:cl-fh}:
\begin{equation} \label{eq:cl-finite-volume}
    \partial_t \bar{u}^h + \partial^h_x r(u) = 0.
\end{equation}
This equation is more commonly written as
\begin{equation} \label{eq:cl-finite-volume-integral}
    h \partial_t \bar{u}^h(x) + r(u)\left(x + \frac{h}{2}\right) - r(u)\left(x - \frac{h}{2}\right) = 0,
\end{equation}
emphasizing that we have two unknown fluxes at the left and right
boundaries of the control volume $[x \pm h / 2]$.

\Cref{eq:cl-finite-volume-integral} is often derived from the integral form of
the conservation law \eqref{eq:conservation-law}:
\begin{equation}
    \frac{\mathrm{d}}{\mathrm{d} t} \int_V u \, \mathrm{d} V
    + \int_{\partial V} r(u) \cdot n \, \mathrm{d} S = 0,
    \quad \forall V \subset \Omega.
\end{equation}
In 1D, with
$V = [x \pm h / 2]$,
$\partial V = \{ x \pm h / 2 \}$, and
$n = \pm 1$,
this is exactly \cref{eq:cl-finite-volume-integral}.

Let $r^h : U \to U$ denote a grid-compatible numerical flux that approximates
the original flux $r$.
By adding the resolved term $\partial^h_x r^h(\bar{u}^h)$ to both sides of
\cref{eq:cl-finite-volume} and rearranging the terms, we get
\begin{equation} \label{eq:fvm}
    L^h(\bar{u}^h) = - \partial^h_x \tau^h(u)
\end{equation}
(see \cref{fig:threepaths}-B), where
$L^h \coloneq \partial_t + \partial^h_x r^h$ is a grid-compatible approximation
of the original conservation law operator $L$ and
\begin{equation} \label{eq:cl-tauh}
    \boxed{\tau^h(u) \coloneq r(u) - r^h(\bar{u}^h)}
\end{equation}
is the residual flux in the FVM equation.
Unlike the LES equation \eqref{eq:les}, the FVM equation \eqref{eq:fvm}
is ``$\partial^h_x$-conservative'', since the conservation law is 
written with respect to the discrete divergence $\partial^h_x$.
The residual flux $\tau^h(u)$ is also known as the \emph{sub-grid} flux.
Calling it a sub-filter flux would be misleading, since
it also depends on the choice of 
numerical divergence $\partial^h_x$ and
numerical flux $r^h$
(in addition to the FVM filter $f^h$).
We use the term \emph{residual flux} for both $\tau^\Delta(u)$ and $\tau^h(u)$.

By approximating the residual flux with a grid-compatible dissipation model
$m^h(\bar{u}^h) \approx \tau^h(u)$, we get the approximate FVM equation
\begin{equation} \label{eq:fvm-approximate}
    L^h(v^h) = -\partial^h_x m^h(v^h)
\end{equation}
(see \cref{fig:threepaths}-C),
where $v^h \approx \bar{u}^h$ is the approximate FVM solution.
In a properly resolved DNS, where $h$ is sufficiently small,
the dissipation model $m^h$ can be set to zero.

The operators $L^h$ and $m^h$ are still continuous and can be evaluated in
every point $x \in \Omega$.
``Fully discrete''
equations for $v^h$ can be
obtained by simply evaluating \cref{eq:fvm-approximate} at the grid points
$x^h_i \coloneq i h$, $i \in \{ 1, \dots, N \}$ on the condition that $h = \ell /
N$ for some integer $N \in \mathbb{N}$.
This restriction is visualized in \cref{fig:threepaths}-B and \cref{fig:threepaths}-C.
The grid-compatible numerical fluxes $r^h$ and $m^h$ ensure that the
restricted equations form a closed system of ordinary differential equations.

While \cref{eq:fvm} can be seen as a discrete LES equation,
the filter width is implied by the grid spacing $h$ used in the discrete
divergence $\partial^h_x$. In LES, we would like to be able to choose the
filter width $\Delta$ independently from $h$. We therefore
propose to combine the LES filter $f^\Delta$ with the FVM filter $f^h$.

\section{Combining LES and the FVM: LES-FVM}
\label{sec:les-fvm}

Classical LES and the FVM both aspire to correctly model the features of the
flow that are larger than a certain length scale.
In LES, the large scales are extracted using an
LES filter $f^\Delta$,
which retains the scales that are larger than the filter width $\Delta$.
In the FVM, the size of the resolved scales are inherently linked to the grid
size $h$ through the FVM filter $f^h$,
numerical flux $r^h$, and discrete divergence $\partial^h_x$.
The fluxes in the LES equation \eqref{eq:les} are defined with respect to the
infinitesimal divergence $\partial_x$,
but the fluxes in FVM equation \eqref{eq:fvm} are
defined with respect to the discrete divergence $\partial^h_x$.

In this section, we present a new discrete LES
formalism---``LES-FVM''---that
bridges the gap between infinitesimal LES and the discrete FVM.
In LES-FVM, the LES filter is applied to selectively extract the scales of
interest, while the FVM filter is applied to make the divergence
grid-compatible.
We first show the classical approach in which one approximates before
discretizing (top route in \cref{fig:threepaths}), and then we propose our new
approach in which we discretize first (middle route in \cref{fig:threepaths}).

\subsection{The classical approach: LES first, then FVM}
\label{sec:cl-classical-discrete-les}

After closure, the approximate LES equations are
$L(v^\Delta) = - \partial_x m^\Delta(v^\Delta)$ (see \cref{fig:threepaths}-G).
Applying the FVM filter $f^h$,
filter-swap property,
adding the resolved term $\partial^h_x r^h(\overline{v^\Delta}^h)$,
and rearranging the terms gives
\begin{equation} \label{eq:classical-discrecte-les-exact}
    L^h\left(\overline{v^\Delta}^h\right) =
    - \partial^h_x \left( m^\Delta(v^\Delta) + \tau^h(v^\Delta) \right)
\end{equation}
(see \cref{fig:threepaths}-H),
where $\overline{v^\Delta}^h$ is the FVM average of the LES approximation
$v^\Delta$.
This equation is unclosed in terms of $\overline{v^\Delta}^h$ since
it contains the field $v^\Delta$ in the right-hand side flux
$m^\Delta(v^\Delta) + \tau^h(v^\Delta)$. Note that $\tau^h(v^\Delta)$ is the
residual FVM flux from \cref{eq:cl-tauh} applied to the LES
approximation $v^\Delta$ instead of $u$.

By using a grid-compatible closure model
$m^{\Delta, h}(\overline{v^\Delta}^h) \approx m^\Delta(v^\Delta) +
\tau^h(v^\Delta)$, we get the approximate equation
\begin{equation} \label{eq:classical-discrecte-les}
    L^h(w^{\Delta, h}) = - \partial^h_x m^{\Delta, h}(w^{\Delta, h})
\end{equation}
(see \cref{fig:threepaths}-I),
where $w^{\Delta, h} \approx \overline{v^\Delta}^h$
is the approximate solution.
We call $m^{\Delta, h}$ a \emph{closure} since it approximates an unresolved
term and thus makes the equation closed. Typically, $m^{\Delta, h}$ is
just a ``discretization'' of the infinitesimal LES closure $m^\Delta$,
but the expression $m^\Delta(v^\Delta) + \tau^h(v^\Delta)$
suggests that $m^{\Delta, h}$ should also account for
the FVM discretization error $\tau^h(v^\Delta)$.
This classical approach involves multiple steps:
close first, then discretize,
and then adjust the closure to account for the
discretization (as visualized in the top route of \cref{fig:threepaths}).

We now show what this classical approach would entail for the 
basic Smagorinsky closure model.
The LES closure
\begin{equation}
    m^\Delta(\bar{u}^\Delta) \coloneq -\theta^2 \Delta^2
    \left| \partial_x \bar{u}^\Delta \right| \partial_x \bar{u}^\Delta
\end{equation}
is used for the LES flux $\tau^\Delta(u)$, while its FVM variant
\begin{equation}
    m^h(\bar{u}^h) \coloneq -\theta^2 h^2
    \left| \partial^h_x \bar{u}^h \right| \partial_x \bar{u}^h,
\end{equation}
is used for the FVM flux $\tau^h(u)$.
A natural choice for $m^{\Delta, h}$ is then
\begin{equation}
    m^{\Delta, h}(\overline{v^\Delta}^h) \coloneq
    -\theta^2 (\Delta^2 + h^2)
    \left| \partial^h_x \overline{v^\Delta}^h \right|
    \partial^h_x \overline{v^\Delta}^h.
\end{equation}
The term
$\theta^2 \Delta^2 |\partial^h_x \cdot | \partial^h_x \cdot$
is a discretization of $m^\Delta$, while
$\theta^2 h^2 |\partial^h_x \cdot | \partial^h_x \cdot$
is a discrete Smagorinsky model for the residual FVM flux $\tau^h$.
The final model $m^{\Delta, h}$ takes the form of a discrete Smagorinsky model 
for a filter of width $\sqrt{\Delta^2 + h^2}$.
The incorporation of $h$ into the filter width can be seen as an adjustment
to account for discretization error $\tau^h(v^\Delta)$.

Interestingly,
if the LES filter $f^\Delta$
is a Gaussian or top-hat filter,
then $\sqrt{\Delta^2 + h^2}$
is the width of the LES-FVM \emph{double filter}
\footnote{
    Note that double filters in LES are also used for dynamically determining
    closure model coefficients by applying a test-filter on top
    of the LES filter~\cite{germanoDynamicSubgridscaleEddy1991}.
    Our double filter serves a different purpose:
    to filter out scales smaller than $\Delta$
    while also coarse-graining the differential operators
    so that they become compatible with a grid of spacing $h$.
}
\begin{equation}
    f^{\Delta, h} \coloneq f^h f^\Delta
\end{equation}
composed of the LES filter $f^\Delta$ and the FVM filter $f^h$.
The field $w^{\Delta, h}$ can therefore be seen as a
Smagorinsky approximation of the LES-FVM solution
$\bar{u}^{\Delta, h} \coloneq f^{\Delta, h} u$.

In the next section,
we propose to approximate $\bar{u}^{\Delta, h}$ directly, using one
single approximation step $v^{\Delta, h} \approx \bar{u}^{\Delta, h}$,
instead of performing successive approximations
$v^\Delta \approx \bar{u}^\Delta$ and
$w^{\Delta, h} \approx \overline{v^\Delta}^h$.

\subsection{Our new approach: LES-FVM}
\label{sec:les-fvm-discretize-first}

We discretize the infinitesimal LES equation
$\overline{L(u)}^\Delta = 0$ by applying the FVM filter $f^h$.
This gives the double filtered conservation law
\begin{equation}
    \overline{L(u)}^{\Delta, h} = 0.
\end{equation}
By using the filter-swap property,
adding the resolved term $\partial^h_x r^h(\bar{u}^{\Delta, h})$ to both
sides, and rearranging the terms,
we get an equation for $\bar{u}^{\Delta, h}$ in
$\partial^h_x$-conservative form:
\begin{equation} \label{eq:cl-les-fvm}
    L^h(\bar{u}^{\Delta, h}) = -\partial^h_x \tau^{\Delta, h}(u)
\end{equation}
(see \cref{fig:threepaths}-E), where
\begin{equation} \label{eq:cl-tau-Delta-h}
    \boxed{\tau^{\Delta, h}(u) \coloneq \overline{r(u)}^\Delta - r^h(\bar{u}^{\Delta, h})}
\end{equation}
is the residual flux in the LES-FVM equation.
This is a novel residual flux expression that accounts for the errors in both
LES and the FVM. The first term only uses the LES filter $f^\Delta$ and the
infinitesimal flux $r$, while the second term uses the LES-FVM filter
$f^{\Delta, h}$ and the numerical flux $r^h$.

Using a discretization-informed grid-compatible closure model
$m^{\Delta, h}(\bar{u}^{\Delta, h}) \approx \tau^{\Delta, h}(u)$ gives
the approximate LES-FVM equation
\begin{equation} \label{eq:cl-discrete-les-model}
    L^h\left(v^{\Delta, h}\right) =
    - \partial^h_x m^{\Delta, h}\left(v^{\Delta, h}\right)
\end{equation}
(see \cref{fig:threepaths}-F),
where $v^{\Delta, h} \approx \bar{u}^{\Delta, h}$
is the approximate LES-FVM solution.
Although the equation for $v^{\Delta, h}$, obtained by
``discretizing first, then approximating'',
has the same form as the equation for $w^{\Delta, h}$ in
\eqref{eq:classical-discrecte-les},
obtained through multiple alternating filtering and approximation steps,
it leads to a different solution because the closure is chosen differently.
In \cref{eq:classical-discrecte-les},
$m^{\Delta, h}$ is chosen such that
$w^{\Delta, h} \approx \overline{v^\Delta}^h$,
which depends on a previous approximation
$v^\Delta \approx \bar{u}^\Delta$.
In \cref{eq:cl-discrete-les-model}, $m^{\Delta, h}$ is chosen such that 
$v^{\Delta, h} \approx \bar{u}^{\Delta, h}$ directly.

Having defined our new LES-FVM framework,
we now proceed to discuss its implications.

\subsection{Control over the LES and FVM filter widths}
\label{sec:filter-widths}

For $f^{\Delta, h}$, we are free to choose $\Delta$ and $h$ independently.
Choosing $\Delta \gg h$ would result in a grid that resolves scales
that are not present in the LES solution $\bar{u}^{\Delta}$, leading to
unnecessary computational costs.
This setting is sometimes referred to as
\emph{explicit LES}, where the effects of the discretization can be neglected.
The advantage is that the classical LES framework can be used with small
discretization errors.
In this case, we have $\tau^{\Delta, h}(u) \approx \tau^\Delta(u)$.
On the other hand, if $\Delta = 0$, we have $f^{\Delta, h} = f^h$, and we
recover the classical FVM setting.
This setting is sometimes referred to as \emph{implicit LES},
where no explicit LES filter $f^\Delta$ is used.
In literature, common choices for the ratio are $0 \leq \Delta / h \leq
4$~\cite{ghosalAnalysisNumericalErrors1996,
geurtsNumericallyInducedHighpass2005}.
In this range, the effect of the FVM filter is equal to or somewhat
smaller than the effect of the LES filter.
It is generally acknowledged that the ratio $\Delta / h$ should ideally be
higher to prevent interference from the discretization in the residual
(for example $\Delta / h \geq 4$ for a second-order accurate discretization
\cite{chowFurtherStudyNumerical2003} or
$\Delta / h \geq 2$ for a sixth-order accurate discretization \cite{vremanLargeeddySimulationTurbulent1997}).

Instead of circumventing the problem of the discretization in LES
by choosing $\Delta \gg h$,
we aim to allow for smaller ratios $\Delta / h$ by
explicitly accounting for the discretization effects in our new residual flux
$\tau^{\Delta, h}(u)$.
In our new LES-FVM framework,
all values of the ratio $\Delta / h$ are valid choices.

\begin{figure*}
    \centering
    \includegraphics[width=0.9\textwidth]{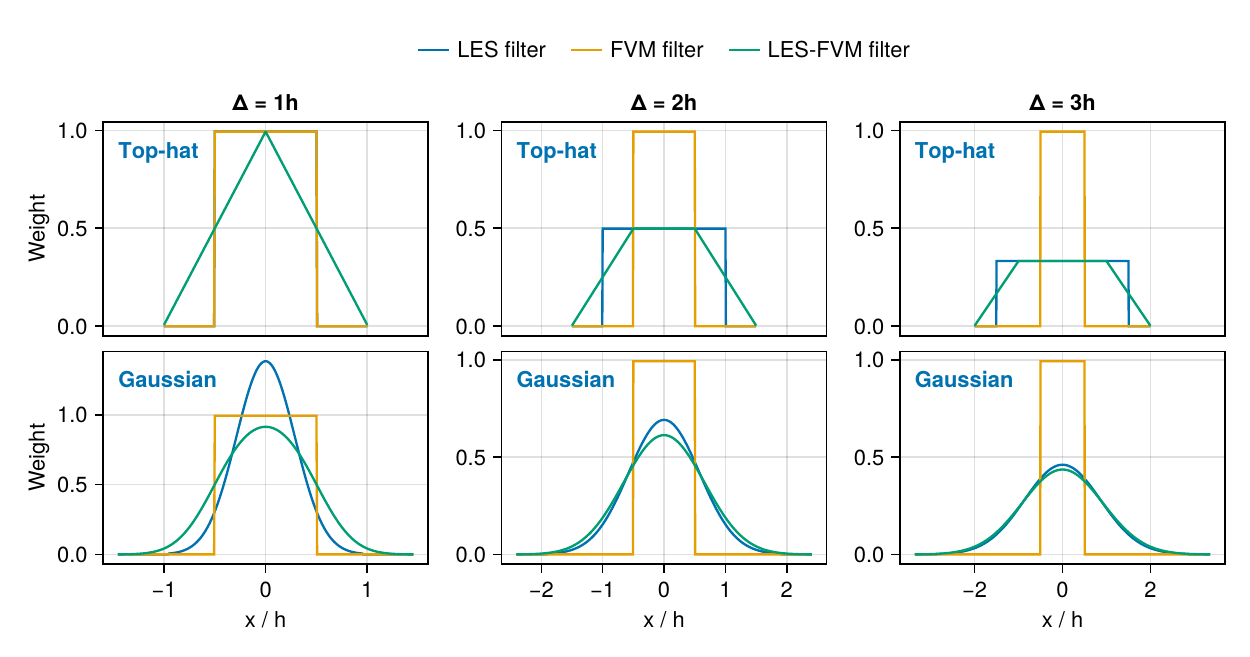} 
    \caption{
        Spatial kernels of LES filter $f^\Delta$ (top-hat and Gaussian),
        FVM grid filter $f^h$ (always top-hat), and
        LES-FVM double filter $f^{\Delta, h}$
        for $\Delta \in \{ 1 h, 2 h, 3 h\}$.
        In the top-left plot, we have $f^\Delta = f^h$.
    }
    \label{fig:filters}
\end{figure*}

In \cref{fig:filters},
we show the filter kernels
$G^\Delta$, $G^h$, and $G^{\Delta, h}$ corresponding to the filters
$f^\Delta$, $f^h$, and $f^{\Delta, h}$ for
top-hat and Gaussian LES kernels
\begin{align}
    G^\Delta_\text{top-hat}(x)
    & \coloneq
    \begin{cases}
        \frac{1}{\Delta} & \text{if } |x| \leq \frac{\Delta}{2}, \\
        0 & \text{otherwise},
    \end{cases}, \\
    \label{eq:gaussian}
    G^\Delta_\text{Gaussian}(x)
    & \coloneq
    \sqrt{\frac{6}{\pi \Delta^2}} \exp\left(-\frac{6 x^2}{\Delta^2}\right).
\end{align}

The LES filter widths are $\Delta \in \{ 1 h, 2 h, 3 h\}$.
For top-hat $G^\Delta$ with $\Delta = h$, we have $G^\Delta = G^h$,
and $G^{\Delta, h}$ becomes a triangular kernel.
For top-hat $G^\Delta$ with $\Delta > h$, the double filter kernel
$G^{\Delta, h}$ takes the shape of a trapezoid,
with $G^{\Delta, h}(x) = G^{\Delta}(x)$ for $|x| \leq (\Delta - h) / 2$.
For Gaussian $G^\Delta$,
the double filter kernel $G^{\Delta, h}$
resembles a Gaussian kernel of width $\sqrt{\Delta^2 + h^2}$.
However, $G^{\Delta, h}$ is not exactly Gaussian.
As the ratio $\Delta / h$ grows, $G^\Delta$ and $G^{\Delta, h}$
become more and more similar.
For the ratios shown, $G^\Delta$ and $G^{\Delta, h}$ are still visually
different.

\begin{figure*}
    \centering
    \includegraphics[width=0.9\textwidth]{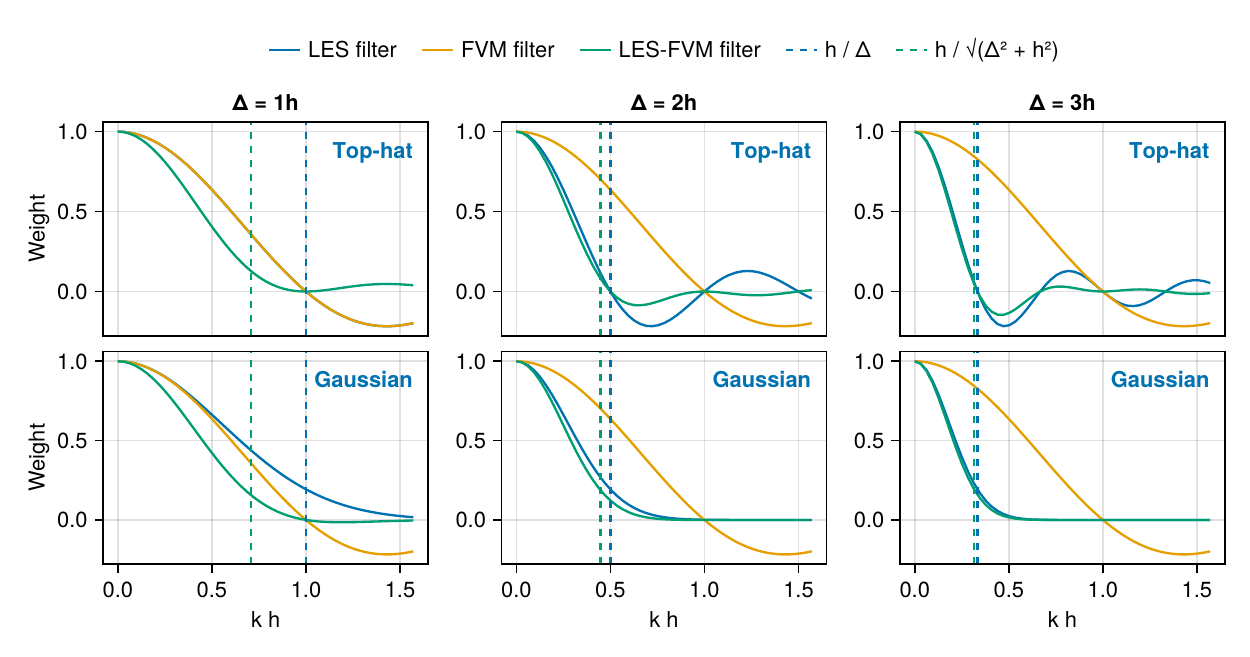} 
    \caption{
        Spectral transfer functions of LES filter $f^\Delta$ (top-hat and Gaussian),
        FVM grid filter $f^h$ (always top-hat), and
        LES-FVM double filter $f^{\Delta, h}$
        for $\Delta \in \{ 1 h, 2 h, 3 h\}$.
        In the top-left plot, we have $f^\Delta = f^h$.
    }
    \label{fig:filter-spectra}
\end{figure*}

For the top-hat and Gaussian LES filters, the spectral transfer functions are
defined as \cite{popeTurbulentFlows2000}
\begin{align}
    \hat{G}^\Delta_\text{top-hat}(\kappa)
    & \coloneq \frac{\sin(\kappa \Delta / 2)}{\kappa \Delta / 2}, \\
    \hat{G}^\Delta_\text{Gaussian}(\kappa)
    & \coloneq \exp\left(-\frac{\kappa^2 \Delta^2}{24}\right),
\end{align}
where $\kappa \coloneq 2 \pi k$ for $k \in \mathbb{Z}$.
Interestingly, they both give the same Taylor series expansion up to second-order around $\kappa = 0$:
\begin{align}
    \hat{G}^\Delta_\text{top-hat}(\kappa)
    & = 1 - \frac{\kappa^2 \Delta^2}{24} + \mathcal{O}(\kappa^4), \\
    \hat{G}^\Delta_\text{Gaussian}(\kappa)
    & = 1 - \frac{\kappa^2 \Delta^2}{24} + \mathcal{O}(\kappa^4).
\end{align}
For the double-filter $f^{\Delta, h}$, we therefore get the expansion
\begin{equation}
    \hat{G}^{\Delta, h}(\kappa)
    = \hat{G}^\Delta(\kappa) \hat{G}^h(\kappa)
    = 1 - \frac{\kappa^2 (\Delta^2 + h^2)}{24} + \mathcal{O}(\kappa^4)
\end{equation}
for both the top-hat and Gaussian LES kernels.
By comparing the second-order terms, we can make the argument that
$\sqrt{\Delta^2 + h^2}$ is indeed the width of the LES-FVM double filter.

The spectral transfer functions $\hat{G}^\Delta$, $\hat{G}^h$, and $\hat{G}^{\Delta, h}$ are shown in
\cref{fig:filter-spectra} as a function of the wavenumber $k \in \mathbb{Z}$.
As the ratio $\Delta / h$ grows, the transfer functions $\hat{G}^\Delta(\kappa)$ and
$\hat{G}^{\Delta, h}(\kappa) = \hat{G}^\Delta(\kappa) \hat{G}^h(\kappa)$ become more and more similar.
Their respective filter widths (represented by vertical dashed lines) also get
closer. For $\Delta = h$, $\hat{G}^{\Delta, h}(\kappa)$
is quite different from $\hat{G}^\Delta(\kappa)$ and $\hat{G}^h(\kappa)$,
and dampens at much lower wavenumbers.

\subsection{Interpretation in terms of existing residual flux expressions}

The residual LES-FVM flux
$\tau^{\Delta, h}$
can be decomposed in two ways in order to relate it to the classical residual
LES flux $\tau^\Delta$. The first decomposition is
\begin{equation}
    \tau^{\Delta, h}(u) = \tau^\Delta(u) + \tau^h(\bar{u}^\Delta),
\end{equation}
meaning that the FVM step induces an additional residual flux
$\tau^h(\bar{u}^\Delta)$ in addition to the residual LES flux
$\tau^\Delta(u)$.
The second decomposition is obtained by writing $\tau^{\Delta, h}$ in terms of
the classical residual LES flux \emph{for the double filter $f^{\Delta, h}$}:
\begin{equation} \label{eq:tau-Delta-h-decomposition}
    \tau^{\Delta, h}(u) =
    \tau^{\Delta, h}_\text{classic}(u) +
    \tau^{\Delta, h}_\text{flux}(u) +
    \tau^{\Delta, h}_\text{div}(u),
\end{equation}
where
\begin{align}
    \tau^{\Delta, h}_\text{classic}(u)
    & \coloneq \overline{r(u)}^{\Delta, h} - r(\bar{u}^{\Delta, h}), \\
    \tau^{\Delta, h}_\text{flux}(u)
    & \coloneq r(\bar{u}^{\Delta, h}) - r^h(\bar{u}^{\Delta, h}), \\
    \tau^{\Delta, h}_\text{div}(u)
    & \coloneq \overline{r(u)}^\Delta - \overline{r(u)}^{\Delta, h}
\end{align}
are the classical residual LES flux for the double filter $f^{\Delta, h}$,
a discretization error from the numerical flux $r^h \approx r$, and
a discretization error from the discrete divergence
$\partial^h_x \approx \partial_x$, respectively.
Due to the filter-swap property,
$\tau^{\Delta, h}_\text{div}(u)$ can be written as a \emph{high-pass grid filter} $1 - f^h$
applied to the LES flux $\overline{r(u)}^\Delta$.
A similar observation was made by Geurts and van der
Bos~\cite{geurtsNumericallyInducedHighpass2005}.
Our residual flux $\tau^{\Delta, h}(u)$ still has two
important differences from the expression of Geurts and van der Bos.
First, our residual flux contains the numerical flux $r^h$ in addition to the
exact flux $r$.
Second, we define the residual flux with respect to the
\emph{discrete} divergence $\partial^h_x$
(through the filter-swap property),
whereas Geurts and van der Bos defined their residual flux with
respect to the \emph{infinitesimal} divergence $\partial_x$
(through the \emph{reverse} filter-swap property).

For the FVM,
Winckelmans~\cite{winckelmansExplicitfilteringLargeeddySimulation2001},
Denaro~\cite{denaroWhatDoesFinite2011,denaroInconsistencefreeIntegralbasedDynamic2018}
and Verstappen~\cite{verstappenMergingFilteringModeling2025} derived a similar
expression for the residual flux as
$\tau^h$ in \cref{eq:cl-tauh},
where only one of the two terms involve the filter.
In our notation, their residual flux resembles $r(u) - r(\bar{u}^h)$, with $r$
instead of $r^h$ in the resolved term.
Our expression $\tau^h$ is different in the sense that it also accounts for the
numerical flux $r^h$. For the Burgers flux \eqref{eq:burgers-discrete}, $r^h$
includes the interpolation and finite difference operators $\eta^h_x$ and
$\partial^h_x$.
Verstappen instead considered the interpolation
$\eta^h_x$ explicitly, as a second filter. He analyzed the filtered equations in
terms of the ``$1 h$-filter'' $f^h$ and the ``$2 h$-filter'' $\eta^h_x f^h$
(their filter widths are $h$ and $2 h$ respectively).
This lead to an orthogonal decomposition of the energy spectrum into three
parts: a resolved part, a sub-$2 h$-filter part, and a sub-$1 h$-filter part.

\subsection{Can we discretize without the FVM filter?}

We can write discrete equations for $\bar{u}^\Delta$ with a resolved and
unresolved part without
applying the FVM filter $f^h$. 
For example, the LES equation $\overline{L(u)}^\Delta = 0$
can be rewritten as
\begin{equation}
    L^h(\bar{u}^\Delta) = -(\partial_x \overline{r(u)}^\Delta - \partial^h_x r^h(\bar{u}^\Delta)).
\end{equation}
The left-hand side is grid-compatible, and if the right-hand side is
approximated by a grid-compatible closure model we do indeed get a discrete
approximation of $\bar{u}^\Delta$ without applying the FVM filter.
However, the residual in the right-hand side cannot be written as a discrete
flux term (unless $f^\Delta$ is the FVM filter $f^h$),
so this equation is not $\partial^h_x$-conservative.
To obtain a discrete conservation law, the FVM filter must be applied.

\subsection{Closure modeling implications}
\label{sec:les-fvm-closure-implications}

In LES, it is common to model the residual flux $\tau^\Delta(u)$
with artificial dissipation, as in eddy viscosity models such as the
Smagorinsky model \eqref{eq:smagorinsky}. For the FVM,
it is common to incorporate artificial numerical dissipation directly into the
numerical flux $r^h$~\cite{dairayNumericalDissipationVs2017},
notably to prevent oscillations around
sharp gradients and shocks~\cite{jamesonConstructionDiscretelyConservative2008}.
When used to model subgrid scale effects, this approach is called
\emph{implicit} LES.
Dissipation can also be explicitly added through a dissipative flux
$m^h(\bar{u}^h) \approx \tau^h(u)$ on top of a
DNS-like flux $r^h(\bar{u}^h)$ that has low numerical dissipation.

In our combined LES-FVM framework,
the closure $m^{\Delta, h}$ should account for both LES and FVM effects.
The advantage of including the numerical flux $r^h$ in the definition of
$\tau^{\Delta, h}$ is that the implicit dissipation produced by $r^h$
is automatically accounted for in $\tau^{\Delta, h}$.
This allows for better determining how much additional dissipation is
required by the explicit closure
$m^{\Delta, h}(\bar{u}^{\Delta, h}) \approx \tau^{\Delta, h}(u)$
\cite{furebyLargeEddySimulation2002}.
If $m^{\Delta, h}$ is a Smagorinsky model, then the Smagorinsky constant can be
tuned so that $m^{\Delta, h}(\bar{u}^{\Delta, h})$ and $\tau^{\Delta, h}(u)$
have similar dissipation profiles for a few reference snapshots $u$ obtained
from DNS.
If $r^h$ is a low-dissipation flux, such as the central difference based
Burgers flux \eqref{eq:burgers-discrete}, the resulting Smagorinsky constant
would be higher.
If $r^h$ already contains implicit numerical dissipation
(as is common for upwinding fluxes), then the resulting Smagorinsky constant
would be lower, since less additional explicit dissipation would be required.

In the classical LES-FVM approach ``approximate first, then discretize'',
the LES closure $m^\Delta$ is in principle made without knowledge of the
discretization effects. As a remedy, information about the discretization is
loosely reinjected into the definition of $m^\Delta$ by relating $\Delta$ to
the grid size $h$.

Next, we turn to the Burgers equation to illustrate the consequences of the new
LES-FVM framework.
In \cref{sec:burgers-dns-aid},
we investigate how the definition of the residual flux affects the LES-FVM
equation. In \cref{sec:burgers-smagorinsky}, we investigate how the choice of
``approximate first'' vs ``discretize first'' leads to different closure
models in practice.

\section{LES-FVM for Burgers' equation} \label{sec:burgers}

We first consider the one-dimensional viscous Burgers' equation defined by the
flux \eqref{eq:burgers}.
This equation can be seen a simplified version of the
Navier-Stokes equations, without the pressure term.
In the inviscid limit, the solution forms shocks,
which can cause oscillations with a coarse grid
discretization with a low-dissipation numerical flux such as the one in
\cref{eq:burgers-discrete}~\cite{jamesonConstructionDiscretelyConservative2008}.
These issues also persist in the viscous case if the grid is too coarse.
This is why discretization-aware closure models are needed.

To assess the correctness of the LES-FVM framework,
we require reference solutions to Burgers' equation.
We use DNS to approximate
the reference solution $u$ and all derived quantities
(such as $\bar{u}^{\Delta, h}$ and $\tau^{\Delta, h}(u)$)
on a DNS grid of sufficiently small grid spacing $h_\text{DNS} \ll h$.
In \ref{sec:two-grids},
we show how the quantities in our LES-FVM framework
can be computed from DNS data in a consistent manner
(such that the filter-swap property stills holds).
This step involves discretizing the two infinitesimal filters
$f^\Delta$ and $f^h$ on the DNS grid.
It also adds the constraint that the refinement factor $h / h_\text{DNS}$
must be odd.
Since the notation required to keep track of quantities related
to both the fine $h_\text{DNS}$-grid and coarse $h$-grid becomes quite verbose,
we will keep the infinitesimal notation for the
reference solution $u$ in the following,
even though it is approximated on the
$h_\text{DNS}$-grid in the code.
Note also that $h$ here refers to the coarse LES-FVM grid spacing,
while in \ref{sec:two-grids}, $h$ refers to the fine DNS grid spacing.

\subsection{Simulation setup}

We use the following unitless parameters.
The domain size is $\ell \coloneq 2 \pi$.
The viscosity is $\nu \coloneq 5 \times 10^{-4}$.
The numerical flux is the central difference based flux in
\cref{eq:burgers-discrete}.
The grid spacings are $h \coloneq \ell / N$ and
$h_\text{DNS} \coloneq \ell / N_\text{DNS}$.
We consider one DNS grid size $N_\text{DNS} \coloneq 13500$
and multiple LES grid sizes
$N \in \{ 300, 900, 2700 \}$.
The compression factors are $N_\text{DNS} / N \in \{ 45, 15, 5 \}$ respectively.
These resolutions give an odd refinement factor for
multiple LES grid sizes at a fixed DNS grid size.
The filter-swap compatibility between the DNS and LES grids requires
$N_\text{DNS} / N$ to be an odd integer (see \ref{sec:two-grids}).
The LES filter kernel is the Gaussian \eqref{eq:gaussian}.
Unless otherwise stated, the LES filter width is $\Delta = 2 h$.

The initial conditions for $u$ are prescribed through the Fourier coefficients
\begin{equation}
\begin{split}
    \hat{a}(k) & \coloneq \left(\frac{k}{k_0}\right)^4
    \exp\left(-2 \left(\frac{k}{k_0}\right)^2 + 2 \pi \mathrm{i} \epsilon_k \right),
\end{split}
\end{equation}
where $\widehat{(\cdot)}$ is the Fourier transform,
$k \in \mathbb{Z}$ is the wavenumber,
$\mathrm{i}$ is the imaginary unit,
$k_0 = 10$ is the peak wavenumber,
and $\epsilon_k \sim \mathcal{U}(0, 1)$
is a uniformly sampled random number between $0$ and $1$ for $k \geq 0$ and
$\epsilon_{k} = -\epsilon_{-k}$ otherwise.
The velocity field is then normalized as
\begin{equation}
    u \coloneq \sqrt{2 E_0} a / \| a \|_\Omega,
\end{equation}
where $E_0 \coloneq 2$ is the chosen initial total energy and
$\| a \|^2_\Omega \coloneq \int_\Omega |a|^2 \, \mathrm{d} x$.
This gives the initial energy spectrum profile
$|\hat{u}(k)|^2 \propto k^4 \mathrm{e}^{-2 (k / k_0)^2}$
which is commonly used for decaying turbulence
problems~\cite{
sanHighorderMethodsDecaying2012,
maulikExplicitImplicitClosures2018,
verstappenMergingFilteringModeling2025}.
In total, we generate $1000$ different initial conditions
$u \in \mathcal{U}_\text{initial}$.

\subsection{DNS-aided LES-FVM}
\label{sec:burgers-dns-aid}

To evaluate the correctness of different residual fluxes, we employ
the ``DNS-aided LES(-FVM)''-framework of Bae and
Lozano-Duran~\cite{baeExactSubgridscaleModels2017,baeNumericalModelingError2022}.
It consists of running a DNS alongside an LES-FVM, where the DNS solution is
used to compute the closure term used in the LES-FVM equation. Ideally, by
using the DNS solution to compute the right-hand side, one would expect to be
able to recover the filtered DNS solution with the DNS-aided LES-FVM.
We here solve the coupled DNS/LES-FVM equations
\begin{equation}
    L(u) = 0, \quad
    L^h(v^{\Delta, h}) = -\partial^h_x \tau^{\Delta, h}_\text{model}(u),
\end{equation}
where $\tau^{\Delta, h}_\text{model}(u)$ is a ``closure'' model
computed from the DNS solution $u$.
This term does not depend on the LES-FVM approximation $v^{\Delta, h}$.
The LES-FVM approximation is initialized with
$v^{\Delta, h} = \bar{u}^{\Delta, h}$.

We test four different expressions for $\tau^{\Delta, h}_\text{model}$:
\begin{align}
    \tau^{\Delta, h}_\text{no-model}(u)
    & \coloneq 0, \\
    \tau^{\Delta, h}_\text{classic}(u)
    & \coloneq
    \overline{r(u)}^{\Delta, h} - r(\bar{u}^{\Delta, h}), \\
    \tau^{\Delta, h}_\text{classic+flux}(u)
    & \coloneq
    \overline{r(u)}^{\Delta, h} - r^h(\bar{u}^{\Delta, h}), \\
    \tau^{\Delta, h}_\text{classic+flux+div}(u)
    & \coloneq
    \overline{r(u)}^\Delta - r^h(\bar{u}^{\Delta, h}).
\end{align}
From \cref{eq:cl-tau-Delta-h} and
the decomposition \eqref{eq:tau-Delta-h-decomposition},
we know that
\begin{equation}
    \tau^{\Delta, h}_\text{classic+flux+div} =
    \tau^{\Delta, h} =
    \tau^{\Delta, h}_\text{classic} +
    \tau^{\Delta, h}_\text{flux} +
    \tau^{\Delta, h}_\text{div}
\end{equation}
is the correct expression.
The two other non-zero models are obtained by neglecting the
discretization errors:
$\tau^{\Delta, h}_\text{classic}$ neglects both the divergence error
$\tau^{\Delta, h}_\text{div}$
and the
numerical flux error
$\tau^{\Delta, h}_\text{flux}$,
while $\tau^{\Delta, h}_\text{classic+flux} =
\tau^{\Delta, h}_\text{classic} + \tau^{\Delta, h}_\text{flux}$ only neglects
the divergence error $\tau^{\Delta, h}_\text{div}$.

Both $u$ and $v^{\Delta, h}$ are advanced forward in time using the
forward-Euler scheme
\begin{align}
    u_{k + 1} & \coloneq u_k - \Delta t_k \partial_x r(u_k) \\
    v^{\Delta, h}_{k + 1} & \coloneq v^{\Delta, h}_k -
    \Delta t_k \partial^h_x \left( r^h(v^{\Delta, h}_k) + \tau^{\Delta, h}_\text{model}(u_k) \right)
\end{align}
where $u_k$ and $v^{\Delta, h}_k$ denote the forward-Euler approximations
to the DNS and LES-FVM solutions at time
$t_k \coloneq \sum_{i = 0}^{k - 1} \Delta t_i$.
The time step
\begin{equation}
    \Delta t_k \coloneq C \times \min \left( \frac{h_\text{DNS}}{\max_{x \in \Omega} | u_k(x) | }, \frac{h_\text{DNS}^2}{\nu} \right)
\end{equation}
is chosen based on the CFL condition for $u_k$ with $C = 0.4$. Both equations
use the same time step. While the time stepping scheme is of first order
accuracy only, it has the advantage of only requiring one evaluation of the
right-hand side per time step. This simplifies the injection procedure, where
the residual flux from DNS is injected into the LES-FVM equation. It
would be more complicated to inject the DNS solution into the LES-FVM
equation for higher order schemes with multiple stages per time step.

\begin{figure}
    \centering
    \includegraphics[width=0.75\columnwidth]{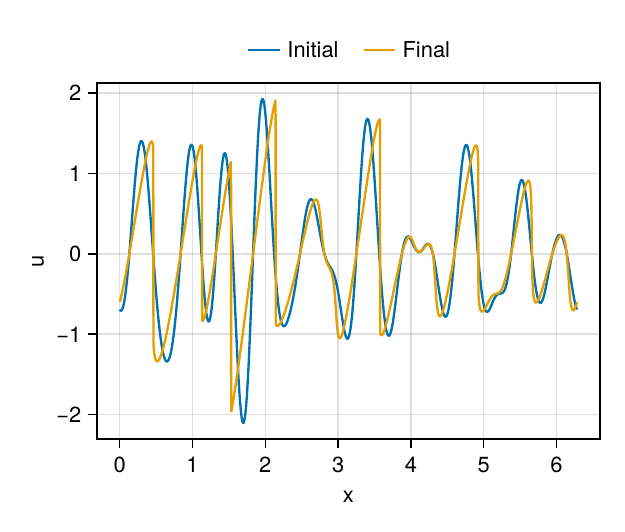} 
    \caption{Initial and final solution to the Burgers equation.}
    \label{fig:burgers_solution}
\end{figure}

In \cref{fig:burgers_solution}, we show the DNS solution for one initial
condition at initial and final time ($t = 0.1$).
The solution is slightly damped due to dissipation,
and some shock-like structures are forming.
These sharp gradients cannot be properly resolved on the coarse $h$-grid
and cause oscillations with the central difference $\partial^h_x$,
which is why a discretization-informed closure model is needed.

\begin{figure}
    \centering
    \includegraphics[width=\columnwidth]{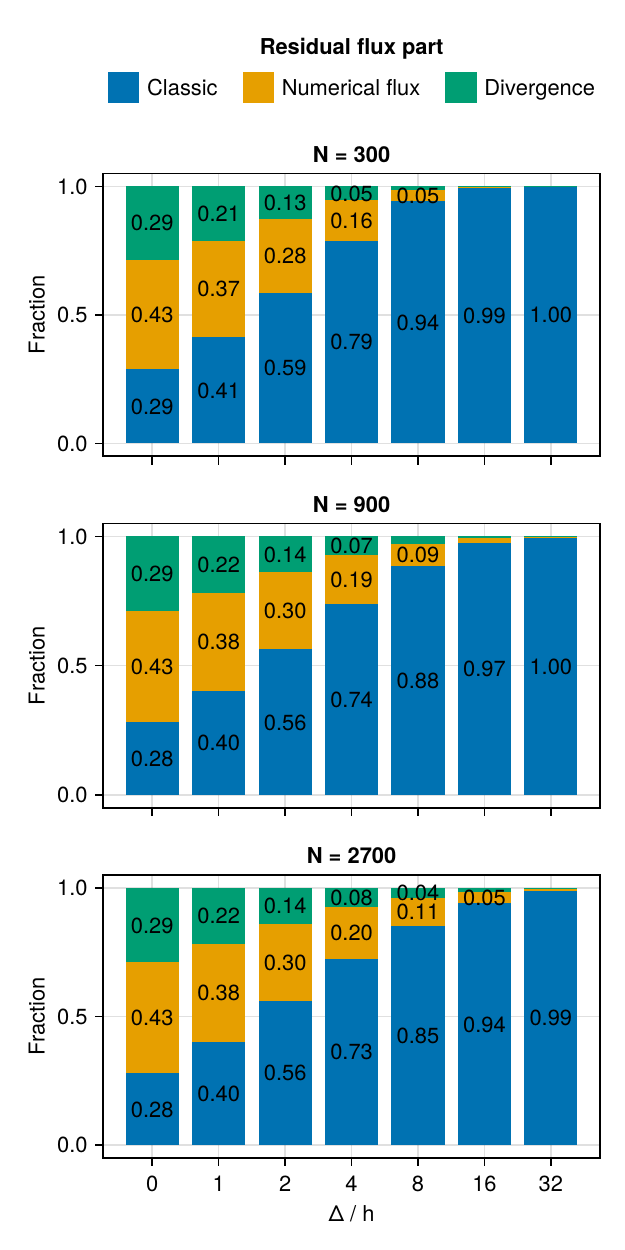} 
    \caption{
        Relative contribution of the different flux parts in the decomposition
        \eqref{eq:tau-Delta-h-decomposition}.
    }
    \label{fig:burgers-fractions}
\end{figure}

In \cref{fig:burgers-fractions}, we show the average fractions
\begin{equation}
    \frac{1}{1000} \sum_{u \in \mathcal{U}_\text{final}}
    \frac{
        \left| \tau^{\Delta, h}_\text{label}(u) \right|
    }{
        \left| \tau^{\Delta, h}_\text{classic}(u) \right| + 
        \left| \tau^{\Delta, h}_\text{flux}(u) \right| + 
        \left| \tau^{\Delta, h}_\text{div}(u) \right|
    }
\end{equation}
for $\text{label} \in \{ \text{classic}, \text{flux}, \text{div} \}$, where
$\mathcal{U}_\text{final}$ contains the DNS solutions at the final time.
For $\Delta = 0$, the classic flux only accounts for about $28 \%$ of the total
residual flux.
For $\Delta = 4 h$, the classic flux accounts for about $73-79 \%$ of the total.
For $\Delta = 32 h$, less than $1 \%$ of the residual is due to the
discretization errors $\tau^{\Delta, h}_\text{flux}$ and
$\tau^{\Delta, h}_\text{div}$.
For explicit LES, at $\Delta \geq 32 h$, we can therefore safely neglect the
discretization errors in the residual flux.
For $\Delta \leq 4 h$, the discretization errors are significant and should
not be neglected.
This corresponds with the observation of Chow and Moin that $\Delta \geq 4 h$
is necessary for the second order
discretization errors to no longer dominate the residual
flux~\cite{chowFurtherStudyNumerical2003}.
For $4 h < \Delta < 32 h$, the discretization errors may or may not be
neglected depending on what accuracy is required. For an actual closure model,
which is going to have a significant modeling error anyway, the $21 \%$
contribution of the discretization to the residual flux at
$(\Delta, N) = (4 h, 300)$ may not be important. But if the residual flux is to
be used as a target for tuning the closure model coefficients, including the
$6 \%$ contribution of the discretization at $(\Delta, N) = (8 h, 300)$ could
still be useful.

\begin{figure*}
    \centering
    \includegraphics[width=\textwidth]{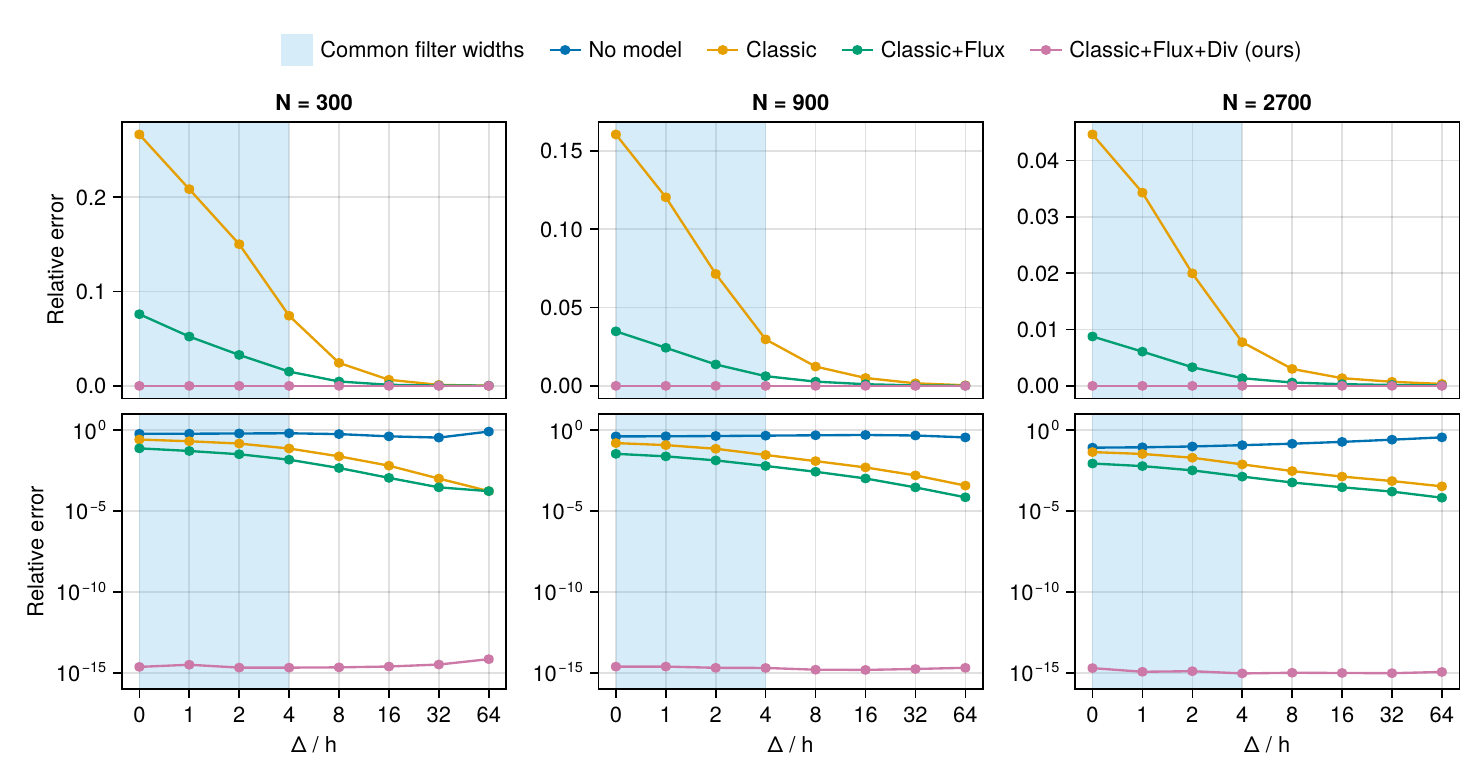}
    \caption{
        Relative errors at final time for DNS-aided LES-FVM of Burgers' equation.
        The shaded area indicates common ratios $\Delta / h$ in literature.
        Top: Linear scale (the no-model error is above the plotted region).
        Bottom: Logarithmic scale.
    }
    \label{fig:burgers-delta-errors}
\end{figure*}

\Cref{fig:burgers-fractions} only shows the importance of the discretization
errors in the a-priori setting.
In \cref{fig:burgers-delta-errors}, we show the average a-posteriori relative errors
\begin{equation}
    e_\text{model} \coloneq \frac{1}{1000} \sum_{u_0 \in \mathcal{U}_\text{initial}}
    \frac{
        \| v^{\Delta, h}_\text{model} - \bar{u}^{\Delta, h} \|_\Omega
    }{
        \| \bar{u}^{\Delta, h} \|_\Omega
    }
\end{equation}
at the final time for different values of $N$ and $\Delta$.
For all $N$ and $\Delta$, $e_\text{classic+flux+div}$ is at machine precision.
For $(N, \Delta) = (300, 0 h)$, we have
$e_\text{no-model} \approx 60 \%$,
$e_\text{classic} \approx 16 \%$, and
$e_\text{classic+flux} \approx 8 \%$.
For $(N, \Delta) = (300, 4 h)$, we have
$e_\text{no-model} \approx 65 \%$,
$e_\text{classic} \approx 7 \%$, and
$e_\text{classic+flux} \approx 2 \%$.
For $(N, \Delta) = (300, 32 h)$, we have
$e_\text{no-model} \approx 35 \%$,
$e_\text{classic} \approx 0.1 \%$, and
$e_\text{classic+flux} \approx 0.03 \%$.
This confirms that the discretization errors are significant for
$\Delta \leq 4 h$, and insignificant for $\Delta \geq 32 h$.

\begin{figure}
    \centering
    \includegraphics[width=\columnwidth]{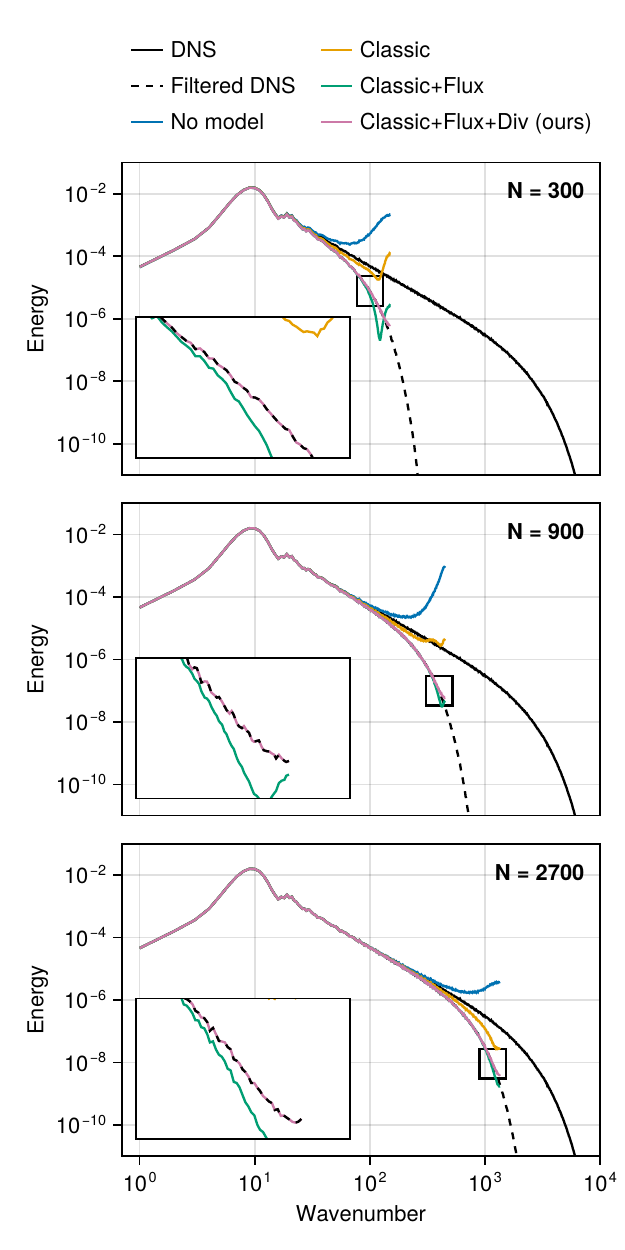}
    \caption{
        Energy spectra for the DNS-aided LES-FVM of Burgers equation.
        Our discretization-informed expression leads to an exact
        correspondence with the filtered DNS, while
        the classical expression is not dissipative enough.
    }
    \label{fig:burgers_spectrum}
\end{figure}

Further insight in the four models is obtained through the energy
spectrum.
We define the energy spectrum of a field $u$ as $|\hat{u}(k)|^2 / 2$, where
$\hat{u}(k)$ is the Fourier transform of $u$ at wavenumber $k$.
\Cref{fig:burgers_spectrum} shows the energy spectra at the final time,
averaged over the $1000$ solutions. Individual DNS solutions have noisy spectra
that fluctuate around the theoretical slope of $k^{-2}$ in the inertial range.
The averaged DNS spectrum is smooth, and adheres to the theoretical slope
for the inertial range.
The DNS grid spacing is small enough to resolve
both the dissipation and inertial ranges.

The filtered DNS spectrum stays on top of the DNS spectrum for the low
wavenumbers, and becomes damped for the highest wavenumbers.
This is because the transfer function of the filter $f^{\Delta, h}$
is close to $1$ for low wavenumbers, but decays for larger wavenumbers
(see \cref{fig:filter-spectra}).
The filtered DNS spectra are shown beyond the cut-off wavenumbers
$N / 2$ of the coarse $h$-grids.
These spectra are obtained by evaluating $\bar{u}^{\Delta, h}(x)$
\emph{at all the DNS grid points $x$},
and then applying the Fourier transform on the DNS grid.
Since $f^{\Delta, h}$ is not a spectral cut-off filter,
the spectrum of $\bar{u}^{\Delta, h}(x^h_i)$
restricted to the coarse $h$-grid points $x^h_i$,
computed with a coarse $h$-grid discrete Fourier transform,
is slightly different from the spectrum of $\bar{u}^{\Delta, h}(x)$ computed on the DNS grid.
In the more detailed zoom-in box, we therefore show the spectrum of
$\bar{u}^{\Delta, h}(x^h_i)$,
which is not defined beyond the cut-off wavenumber $N / 2$ of the
$h$-grid.

The LES-FVM spectra are all evaluated on the $h$-grid, and they are therefore
not defined beyond $N / 2$.
For all three grid sizes, the no-closure solution has too much energy in the
highest resolved wavenumbers. This is common in LES, and part of the motivation
for using a dissipative closure model.
The classic flux $\tau^{\Delta, h}_\text{classic}$ is not dissipative enough.
For the first two resolutions, it even shows signs of instability in the
highest wavenumbers.
The improved flux $\tau^{\Delta, h}_\text{classic+flux}$ gives a spectrum
closer to the filtered DNS than $\tau^{\Delta, h}_\text{classic}$,
but at the highest resolved wavenumbers, it is generally
\emph{too dissipative}.
It also shows a sign of instability at the highest wavenumbers.
This is likely because the DNS-aided closure term
$\tau^{\Delta, h}_\text{classic+flux}(u)$
does not depend on the LES-FVM solution $v^{\Delta, h}_\text{classic+flux}$ at all.
As $v^{\Delta, h}_\text{classic+flux}$ starts to decorrelate from $\bar{u}^{\Delta, h}$,
the DNS-aided closure term becomes less accurate over time and even detrimental
at the highest wavenumbers.
The correct flux $\tau^{\Delta, h}_\text{classic+flux+div}$ gives a spectrum
that perfectly overlaps with the $h$-grid variant of the filtered DNS
spectrum (in the zoom-in box).

\paragraph{Conclusion}
Using the correct expression for the residual LES-FVM flux as a closure
model gives \emph{perfect} results, unlike the classical residual LES flux
expression.
The exact residual flux $\tau^{\Delta, h}_\text{classic+flux+div}(u) = \tau^{\Delta, h}(u)$
can therefore be seen as the
\emph{best case scenario} for an LES-FVM closure model
$m^{\Delta, h}(\bar{u}^{\Delta, h})$.
However, in general $\tau^{\Delta, h}(u)$ cannot
be computed from $\bar{u}^{\Delta, h}$ alone,
since information is lost in the filtering process.

Next, we turn to an actual LES-FVM closure model that does not require a parallel
DNS simulation for evaluation. The question is whether the
discretization-informed expression for the residual flux can actually be
used to make better closure models.

\subsection{Smagorinsky closure}
\label{sec:burgers-smagorinsky}

Define the infinitesimal and discrete Smagorinsky fluxes as
\begin{equation}
    s(u) \coloneq \left| \partial_x u \right| \partial_x u, \quad
    s^h(u) \coloneq \left| \partial^h_x u \right| \partial^h_x u.
\end{equation}
For the double-filter $f^{\Delta, h}$ of width $\sqrt{\Delta^2 + h^2}$,
the classical LES Smagorinsky model is defined at the infinitesimal level as
\begin{equation}
    m^{\Delta, h}_\text{classic}(u; \theta)
    \coloneq -\theta^2 (\Delta^2 + h^2) s(u),
\end{equation}
where $\theta$ is a model parameter.
This classical LES model is then discretized as
\begin{equation}
m^{\Delta, h}_\text{discrete}(u; \theta) \coloneq
-\theta^2 (\Delta^2 + h^2) s^h(u).
\end{equation}

The coefficient $\theta$ can be estimated in various ways.
For the infinitesimal basic Smagorinsky model (before discretizing the flux),
a value can be derived analytically~\cite{lillyRepresentationSmallscaleTurbulence1966}.
In the dynamic Smagorinsky model, the coefficient is estimated from 
the LES state itself using a hypothesis of scale-similarity combined with a
test-filter~\cite{germanoDynamicSubgridscaleEddy1991}.
We employ the basic Smagorinsky framework, but with a data-driven estimate of
the coefficient using a least-squares fit to the residual flux.

In the classical LES framework,
the infinitesimal model
$m^{\Delta, h}_\text{classic}(\bar{u}^{\Delta, h}; \theta)$
is used to approximate the residual LES flux
$\tau^{\Delta, h}_\text{classic}(u) \coloneq
\overline{r(u)}^{\Delta, h} - r(\bar{u}^{\Delta, h})$.
The discretization is then only accounted for in the filter definition.
For a set of reference solution snapshots $\mathcal{U}$ (obtained from DNS),
we therefore estimate the coefficient for the classical Smagorinsky model as
\begin{equation}
    \min_\theta \
    \sum_{u \in \mathcal{U}}
    \left| -\theta^2 (\Delta^2 + h^2)
    s(\bar{u}^{\Delta, h}) - \tau^{\Delta, h}_\text{classic}(u) \right|^2,
\end{equation}
where the infinitesimal operators are approximated on the DNS grid.
Similarly, in the new LES-FVM framework, a discretization-informed Smagorinsky
model can be obtained by choosing $\theta$ such that
$m^{\Delta, h}_\text{discrete}(\bar{u}^{\Delta, h}; \theta)$ matches the
discretization-informed residual flux
$\tau^{\Delta, h}(u) \coloneq
\overline{r(u)}^\Delta - r^h(\bar{u}^{\Delta, h})$ as
\begin{equation}
    \min_\theta \
    \sum_{u \in \mathcal{U}}
    \left| -\theta^2 (\Delta^2 + h^2)
    s^h(\bar{u}^{\Delta, h}) - \tau^{\Delta, h}(u) \right|^2.
\end{equation}
We thus get the two different least-squares estimates for $\theta$:
\begin{align}
    -\theta^2_\text{classic} (\Delta^2 + h^2)
    & =
    \frac{
        \sum_{u \in \mathcal{U}}
        \left\langle
        s(\bar{u}^{\Delta, h}),
        \tau^{\Delta, h}_\text{classic}(u)
        \right\rangle
    }{
        \sum_{u \in \mathcal{U}}
        \left\langle
        s(\bar{u}^{\Delta, h}),
        s(\bar{u}^{\Delta, h})
        \right\rangle
    }, \\
    -\theta^2_\text{informed} (\Delta^2 + h^2)
    & = 
    \frac{
        \sum_{u \in \mathcal{U}}
        \left\langle
        s^h(\bar{u}^{\Delta, h}),
        \tau^{\Delta, h}(u)
        \right\rangle
    }{
        \sum_{u \in \mathcal{U}}
        \left\langle
        s^h(\bar{u}^{\Delta, h}),
        s^h(\bar{u}^{\Delta, h})
        \right\rangle
    },
\end{align}
where $\langle \cdot, \cdot \rangle$ denotes the inner product on $\Omega$.
Solving for $\theta$ is then straightforward.
These expressions are similar to the ones in the dynamic Smagorinsky model
(where the coefficient is also tuned using a least-squares estimate),
but here we use DNS data $u$ instead of LES data to compute the target residual flux.

What is the difference between $\theta_\text{classic}$ and $\theta_\text{informed}$?
In the classical LES setting, $\theta_\text{classic}$ is obtained \emph{before}
discretization (information about the discretization is only injected through
the definition of the filter width).
In our framework, $\theta_\text{informed}$ is obtained \emph{after}
discretization.
It should therefore be slightly higher than $\theta_\text{classic}$ to account
for the additional dissipation required by the discretization.
Our hypothesis is that the classical discretized model
$m^{\Delta, h}_\text{discrete}(\cdot; \theta_\text{classic})$
should perform worse than the discretization-informed model
$m^{\Delta, h}_\text{discrete}(\cdot; \theta_\text{informed})$,
since the choice of
$\theta_\text{classic}$ assumes an infinitesimal setting,
while $\theta_\text{informed}$ is obtained in the discrete setting,
using the correct discrete residual flux as a target.

\begin{figure}[t]
    \centering
    \includegraphics[width=0.9\columnwidth]{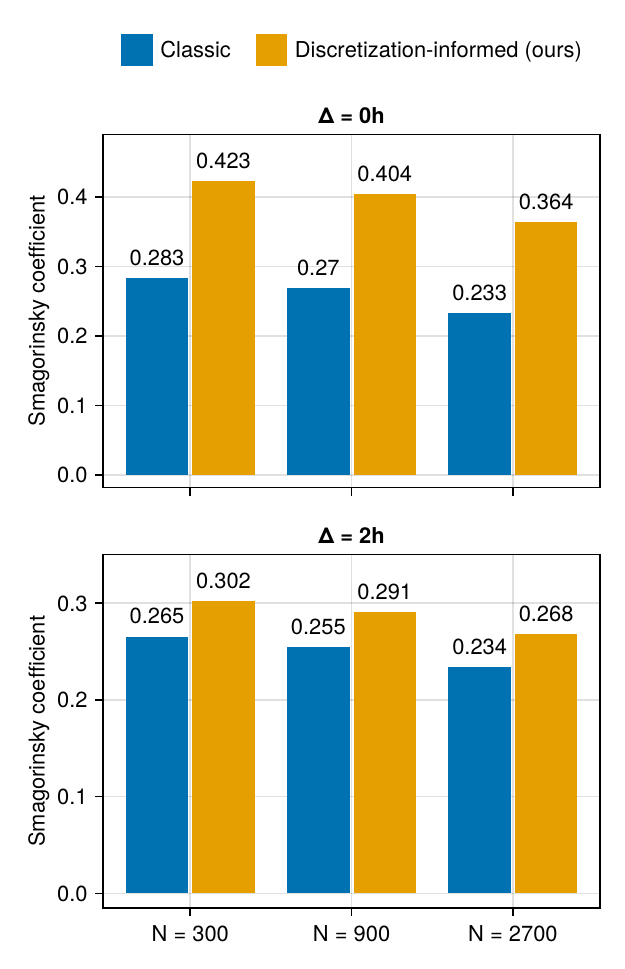}
    \caption{Fitted Smagorinsky coefficients for Burgers' equation.}
    \label{fig:burgers_smagorinsky_coefficients}
\end{figure}

In
\cref{fig:burgers_smagorinsky_coefficients},
we show the Smagorinsky coefficients
estimated from DNS data for Gaussian $f^\Delta$ and
$\Delta \in \{ 0 h,  2 h \}$. Indeed,
$\theta_\text{discrete}$ is larger than
$\theta_\text{classic}$ for all grid sizes.
The two coefficients are more different for $\Delta = 0 h$ than for $\Delta = 2 h$,
likely because
$\tau^{\Delta, h}$ converges to $\tau^{\Delta, h}_\text{classic}$
when the ratio $\Delta / h$ increases (see \cref{fig:burgers-fractions}).
For even larger $\Delta$, for example $\Delta = 32 h$,
the two coefficients would likely be very similar.

Since all three grid sizes lay within the inertial range of the energy spectrum,
we could also choose to fit one single Smagorinsky coefficient for all grid sizes.
Currently, we fit separate coefficients for each grid size $N$ and filter ratio
$\Delta / h$, resulting in larger
coefficients for the smaller grid sizes and filter ratios.
This is likely because the sharp gradients are less
resolved, and therefore more dissipation is needed.

\begin{figure}[t]
    \centering
    \includegraphics[width=0.9\columnwidth]{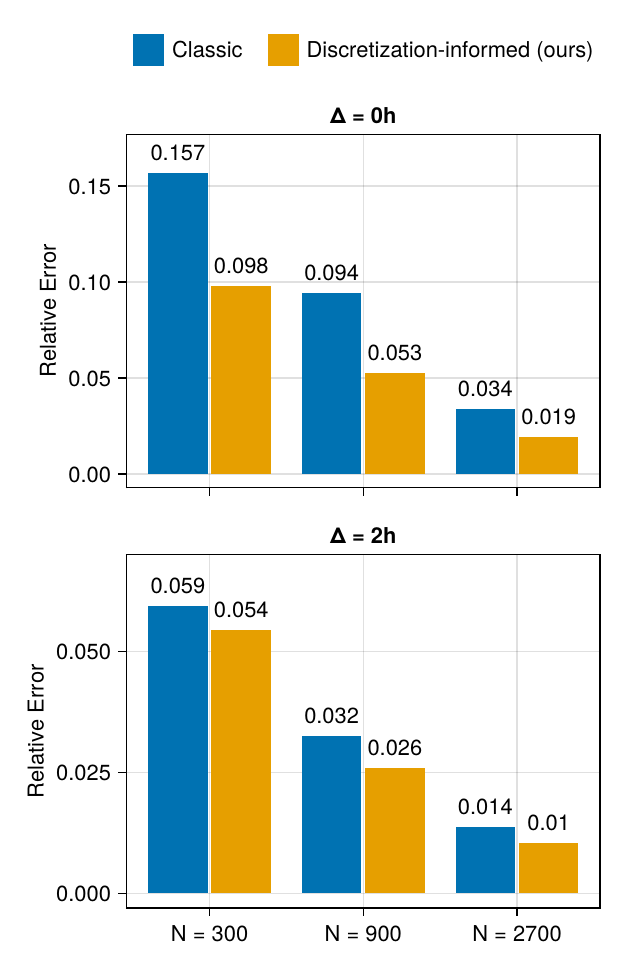}
    \caption{Relative errors at final time for LES-FVM of Burgers' equation with the Smagorinsky model.}
    \label{fig:burgers_smagorinsky_error}
\end{figure}

In \cref{fig:burgers_smagorinsky_error}, we show the relative errors at the
final time between the Smagorinsky LES-FVM approximation and the exact LES-FVM
solution. Our discretization-informed Smagorinsky model has lower errors than
the classical Smagorinsky model. The difference is significant for $\Delta = 0 h$,
but smaller for $\Delta = 2 h$. For even larger filter ratios, the difference
would likely vanish.

\begin{figure*}
    \centering
    \includegraphics[width=0.8\textwidth]{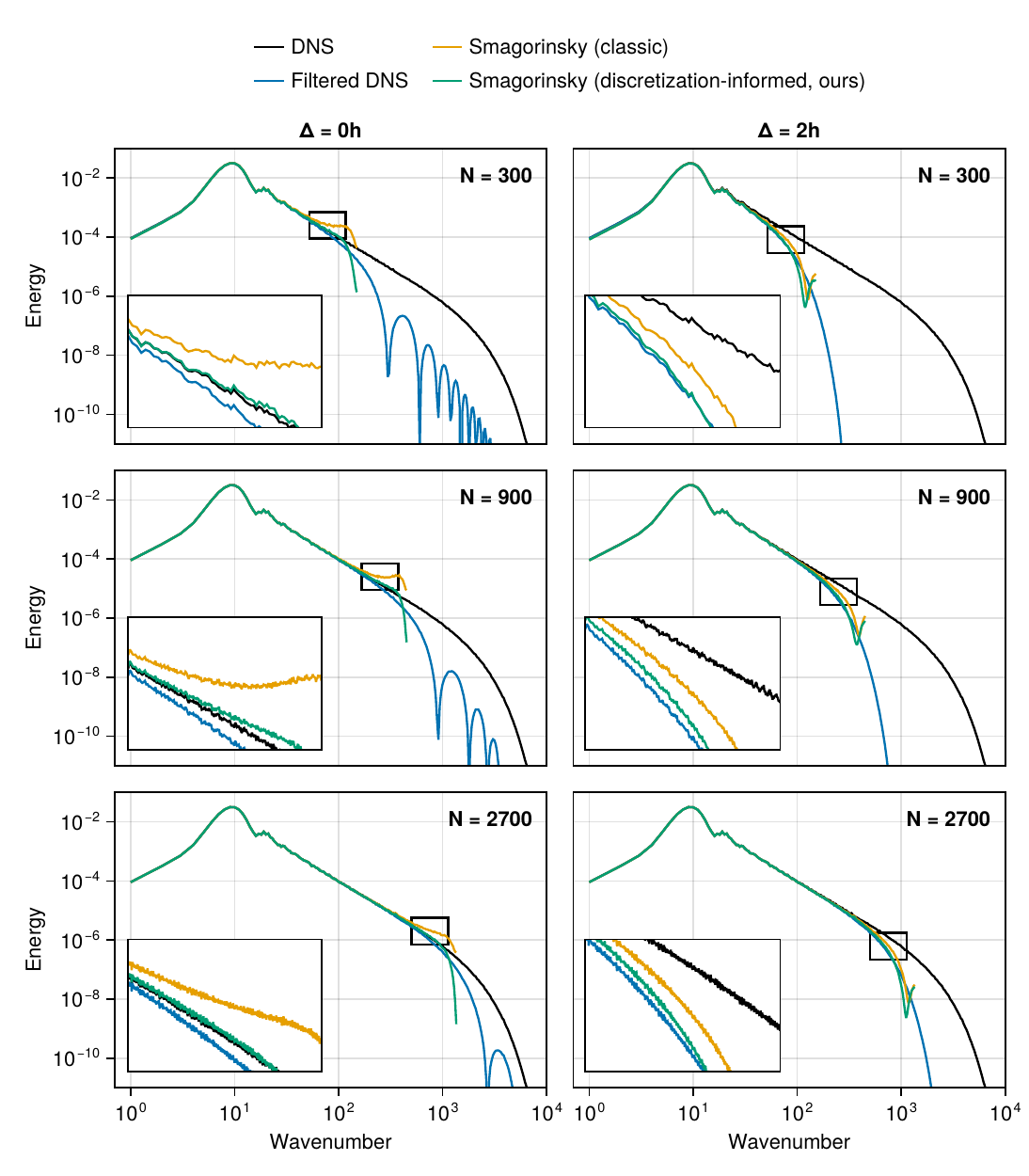}
    \caption{
        Energy spectra for the Burgers equation solved with the Smagorinsky
        model.
    }
    \label{fig:burgers_smagorinsky_spectrum}
\end{figure*}

The difference is further visible in the
energy spectra shown in \cref{fig:burgers_smagorinsky_spectrum}. We see that
the classical Smagorinsky model is not dissipative enough with a larger
spectrum than the filtered DNS. The discretization-informed Smagorinsky model
has a spectrum closer to the filtered DNS. For $\Delta = 2 h$, both models have
a peak at the highest wavenumbers. This is because the central difference
scheme used in $r^h$ gives rise to oscillations around the sharp gradients, and the
dissipative Smagorinsky model is not able to fully remove them without causing
over-dissipation across all wavenumbers. The Smagorinsky coefficient obtained
from the least-squares fit is a compromise between removing these oscillations
and preserving the rest of the spectrum.

We stress that the gain in accuracy obtained by tuning the Smagorinsky
model with the correct residual flux instead of the classic one, is strongly
limited by the eddy viscosity hypothesis. We expect that with more expressive
closure models, such as machine-learning-based ones, much lower errors can be
obtained, and the benefits of our new residual flux expression will be much
more prominent.

In conclusion, we have shown that for common filter ratios, the
discretization-informed residual flux for the 1D Burgers equation both leads to
zero error in a DNS-aided LES and improves the performance of an actual closure
model when used as a target for fitting the model parameters.
Next, we turn to the 3D incompressible Navier-Stokes equations. As we show,
they have additional complexities that must be overcome to achieve a
discretization-informed expression for the residual.

\section{LES-FVM for the 3D incompressible Navier-Stokes equations}
\label{sec:navier-stokes}

The incompressible Navier-Stokes equations in index-form are given by
\begin{equation} \label{eq:navier-stokes}
    \partial_j u_j = 0, \quad
    \partial_t u_i +
    \partial_j \left( \sigma_{i j}(u) + p \delta_{i j} \right) =
    0,
\end{equation}
where
\begin{equation} \label{eq:stress}
    \sigma_{i j}(u) \coloneq u_i u_j - \nu \left(\partial_j u_i + \partial_i u_j\right)
\end{equation}
is the convective-diffusive stress tensor,
$\delta_{i j}$ is the identity tensor (Kronecker delta-symbol),
$(i, j) \in \{1, 2, 3\}^2$ are indices, 
$\partial_t \coloneq \partial / \partial t$ and
$\partial_i \coloneq \partial / \partial x_i$ are partial derivatives,
$x \coloneq (x_1, x_2, x_3)$ is the position,
$t$ is the time,
$u_i(x, t)$ is the velocity in direction $i$,
$p(x, t)$ is the pressure,
and $\nu > 0$ is the kinematic viscosity.
Repeated indices imply summation (Einstein notation).
Since our derivations involve mainly spatial derivatives,
we will write $u_i(x)$ instead of $u_i(x, t)$
to ease the notation.
For simplicity, we assume the equations are defined in a periodic box
$\Omega = [0, \ell]^3$ of side length $\ell > 0$.

The space of periodic scalar-valued fields on $\Omega$ is denoted $U$.
Although the entries of the tensor
$\sigma(u) \in U^{3 \times 3}$ are scalar fields
$\sigma_{i j}(u) \in U$, we still use the word ``tensor'' to describe
both $\sigma(u)$ and $\sigma_{i j}(u)$ interchangeably
(and similarly for other vector and tensor fields).

\subsection{Pressure-free Navier-Stokes equations}

The system \eqref{eq:navier-stokes} does not have the form of a
conservation law for the state $(u, p) \in U^4$. To repeat the
procedure from the previous sections (1D), we first need to
rewrite the Navier-Stokes equations as a self-contained conservation-law. To
achieve this, we will eliminate the pressure $p$ from the momentum equations.

The pressure $p$ is a Lagrange multiplier that enforces the continuity equation
for $u$. A related viewpoint is that $p \delta_{i j}$ is a correction that makes the
non-divergence-preserving stress tensor $\sigma_{i j}(u)$ divergence-preserving.
We say that a stress tensor $\sigma \in U^{3 \times 3}$ is divergence-preserving
if $\partial_i \partial_j \sigma_{i j} = 0$,
i.e.\ the force $\partial_j \sigma_{i j}$ is divergence-free.

In \ref{sec:projection}, we introduce the two pressure projection
operators
\begin{equation}
    \pi_{i j} \coloneq \delta_{i j} -
    \partial_i \left(\partial_k \partial_k\right)^{\dagger} \partial_j
\end{equation}
and 
\begin{equation}
    \pi_{i j \alpha \beta} \coloneq
    \delta_{i \alpha} \delta_{j \beta} -
    \delta_{i j} \left(\partial_k \partial_k\right)^{\dagger} \partial_\alpha
    \partial_\beta,
\end{equation}
where $(\partial_k \partial_k)^{\dagger}$
is the inverse Laplacian operator subject to the constraint
of an average pressure of zero.
The vector-projector $\pi_{i j}$ can be used to make vector fields
divergence-free (since $\partial_i \pi_{i j} = 0$), while 
the tensor-projector
$\pi_{i j \alpha \beta}$ can be used to make tensor fields
divergence-\emph{preserving} (since $\partial_i \partial_j \pi_{i j \alpha \beta} = 0$).
The proofs are given in \cref{th:vector-projector} and \cref{th:stress-projector}.

Since $u$ is divergence-free, we can rewrite the
momentum equation in a pressure-free way
using either of the two projectors as $\partial_t u_i + \pi_{i
j} \partial_k \sigma_{j k}(u) = 0$ or
$\partial_t u_i + \partial_j \pi_{i j \alpha \beta} \sigma_{\alpha \beta}(u) = 0$.
We will use the latter form, since it is has the form of a conservation law
(divergence of a tensor).
The pressure-free momentum equations can thus be written as
\begin{equation} \label{eq:momentum-projected}
    L_i(u) \coloneq \partial_t u_i + \partial_j r_{i j}(u) = 0,
\end{equation}
where
\begin{equation}
    r_{i j}(u) \coloneq \pi_{i j \alpha \beta} \sigma_{\alpha \beta}(u)
    = \sigma_{i j}(u) + p \delta_{i j}
\end{equation}
is the projected stress tensor and $L_i \coloneq \partial_t + \partial_j r_{i j}$
is the pressure-free momentum
equation operator. Given $\sigma(u)$,
the projector $\pi$ computes the unique pressure
$p$ (up to a constant) such that $\sigma(u) + p \delta$ is
divergence-preserving. The pressure projection only modifies the
diagonal, so $r_{i j} = \sigma_{i j}$ for $i \neq j$.

In \cref{eq:navier-stokes},
there is a spatial constraint of divergence-freeness, which was not present
in \cref{sec:conservation-law}. In the projected form
\eqref{eq:momentum-projected}, this constraint is hidden inside $r_{i j}(u)$.
As a result, the 3D stress tensor $r_{i j}(u)$ is non-local in $u$,
and requires solving a Poisson equation,
whereas the 1D Burgers flux $r(u)$ was local.
As long as the initial velocity field is divergence-free,
the continuity equation can be ignored since
\cref{eq:momentum-projected} evolves the velocity field in a
divergence-preserving way.

Since we incorporated the divergence-free constraint,
the 3D conservation law \eqref{eq:momentum-projected}
has the same form as
the 1D conservation law \eqref{eq:conservation-law}.
We can therefore repeat the procedure from \cref{sec:les-fvm}
to obtain LES-FVM equations in conservative form.
We use the same notation as in
\cref{sec:conservation-law,sec:les-fvm,sec:burgers}
to highlight the similarities and differences.
One difference is the presence of direction indices $i$ and $j$.
The scalar flux $r(u)$ from \cref{sec:conservation-law}
is now a $3 \times 3$ stress tensor $r_{i j}(u)$.

\subsection{Classical LES}

Consider a convolutional homogeneous spatial filter
$f^\Delta : u \mapsto \bar{u}^\Delta$ defined
for all scalar fields $u \in U$ as
\begin{equation}
    \bar{u}^\Delta(x) \coloneq \int_{\mathbb{R}^3} G^\Delta(x - y) u(y) \, \mathrm{d} y
\end{equation}
for some kernel $G^\Delta$.
We integrate over $\mathbb{R}^3$ instead of $\Omega$
to allow for periodic extension.
As in 1D, this filter commutes with differentiation:
\begin{equation}
    f^\Delta \partial_i = \partial_i f^\Delta.
\end{equation}
The filtered Navier-Stokes equations therefore take the conservative form
\begin{equation} \label{eq:navier-stokes-filtered}
    \partial_j \bar{u}^\Delta_j = 0, \quad
    \partial_t \bar{u}^\Delta_i
    + \partial_j \left( \sigma_{i j}(\bar{u}^\Delta) + \xi^\Delta_{i j}(u) + \bar{p}^\Delta \delta_{i j} \right)
    = 0,
\end{equation}
where $\bar{u}^\Delta_i$ and $\bar{p}^\Delta$ are filtered fields and
\begin{equation}
    \xi^\Delta_{i j}(u) \coloneq \overline{u_i u_j}^\Delta - \bar{u}^\Delta_i \bar{u}^\Delta_j
\end{equation}
is the classical RST.
We reserve the symbol $\tau^\Delta$ for the RST
$\tau^\Delta(u) \coloneq \pi \xi^\Delta(u) = \overline{r(u)}^\Delta - r(\bar{u}^\Delta)$,
which contains the pressure projector $\pi$.
Since filtering commutes with differentiation,
we also have $f^\Delta \pi = \pi f^\Delta$.

For classical structural LES models, the unprojected tensor
$\xi^\Delta_{i j}(u)$ is
replaced by a closure model $m^\Delta_{i j}(\bar{u}^\Delta)$
that only depends on $\bar{u}^\Delta$.
Since $\xi^\Delta_{i j}$ is a symmetric tensor ($\xi^\Delta_{i j} = \xi^\Delta_{j i}$),
the closure model is designed to be symmetric as well ($m^\Delta_{i j} = m^\Delta_{j i}$).
If the filter $f^\Delta$ is local in space,
then $\xi^\Delta$ is also local, and $m^\Delta$ can be chosen to be a local
closure.
To solve the LES equations, the closed equations are discretized.
Since $m^\Delta$ is symmetric, its discretized variant is also symmetric.

\subsection{Classical FVM}

Using the staggered spatial discretization scheme of Harlow and
Welch~\cite{harlowNumericalCalculationTimeDependent1965},
we define the DNS equations as
\begin{equation} \label{eq:ns-dns}
    \partial^h_j u^h_j = 0, \quad
    \partial_t u^h_i
    + \partial^h_j \left( \sigma^h_{i j}\left(u^h\right) + p^h \delta_{i j} \right)
    = 0.
\end{equation}
Here, $u^h \in U^3$ and $p^h \in U$ are the DNS velocity and pressure fields,
\begin{equation}
    \sigma^h_{i j}(u) \coloneq \Big(\eta^h_j u_i\Big) \Big(\eta^h_i u_j\Big) -
    \nu \left(\partial^h_j u_i + \partial^h_i u_j\right)
\end{equation}
is a grid-compatible numerical stress tensor
analogous to the 1D numerical Burgers flux $r^h$
from \cref{eq:burgers-discrete},
$\partial^h_i : U \to U$ and $\eta^h_i : U \to U$ are finite difference and
interpolation operators defined for all $u \in U$ as
\begin{align}
    \partial^h_i u(x) & \coloneq
    \frac{1}{h} \left[
        u\left(x + \frac{h}{2} e_i\right) -
        u\left(x - \frac{h}{2} e_i\right)
    \right], \\
    \eta^h_i u(x) & \coloneq
    \frac{1}{2} \left[
        u\left(x - \frac{h}{2} e_i\right) +
        u\left(x + \frac{h}{2} e_i\right)
    \right],
    \label{eq:eta}
\end{align}
where $h$ is a uniform grid spacing and $(e_i)_{i = 1}^3$ are the unit vectors.
These operators are second-order accurate in $h$, and, as a result,
$u^h(x, t)$ is a second-order accurate approximation of $u(x, t)$
for all $x$ and $t$ if $u^h(x, 0) = u(x, 0)$.

We use the projected form of the DNS equations \eqref{eq:ns-dns}:
\begin{equation} \label{eq:ns-dns-projected}
    L^h_i(u^h) \coloneq
    \partial_t u^h_i
    + \partial^h_j r^h_{i j}(u^h)
    = 0,
\end{equation}
where $L^h_i$ is the pressure-free finite volume momentum equation operator,
\begin{equation}
    r^h_{i j} \coloneq \pi^h_{i j \alpha \beta} \sigma^h_{\alpha \beta}
\end{equation}
is the projected discrete stress tensor, and
\begin{equation}
    \pi^h_{i j \alpha \beta} \coloneq
    \delta_{i \alpha} \delta_{j \beta} -
    \delta_{i j} \left(\partial^h_k \partial^h_k\right)^{\dagger} \partial^h_\alpha
    \partial^h_\beta
\end{equation}
is a discrete version of $\pi_{i j \alpha \beta}$ that makes stress tensors
discretely divergence-preserving
(i.e.\ $\partial^h_i \partial^h_j \pi^h_{i j \alpha \beta} = 0$).
Given the stress $\sigma^h_{i j}(u^h)$, the projector
$\pi^h_{i j \alpha \beta}$ computes the isotropic pressure correction
$p^h \delta_{i j}$ that makes $r^h_{i j}(u^h)$ divergence-preserving, so we get the
identity $r^h_{i j}(u^h) = \sigma^h_{i j}(u^h) + p^h \delta_{i j}$.

\begin{figure}
    \centering
    \includegraphics[width=\columnwidth]{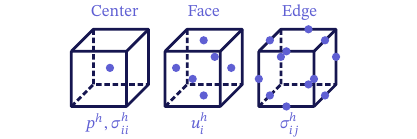}
    \caption{
        Staggered positions in a reference volume of
        a scalar $p^h$, vector $u^h$, and tensor $\sigma^h$.
        The diagonal tensor components are in the volume center, while
        the off-diagonal tensor components are on the volume edges.
    }
    \label{fig:staggered}
\end{figure}

We divide the domain $\Omega = [0, \ell]^3$ into $N \coloneq \ell / h$
reference volumes in each dimension ($N^3$ volumes in total).
When restricting the FVM solution, we use a staggered representation as depicted
in \cref{fig:staggered}. The pressure $p^h$ is restricted to the volume centers
(pressure points). The velocity components $u^h_i$ are restricted to the centers
of the volume faces orthogonal to $e_i$ (velocity points).
The positioning of the tensor components $\sigma^h_{i j}$ follows naturally.
They are in the pressure points if $i = j$, and in the
centers of the volume edges otherwise.
The continuity equation is evaluated in the pressure points, and the momentum
equations in the velocity points.

Note that while we still use the continuous notation $u^h(x, t)$, the degrees of
freedom depicted in \cref{fig:staggered} contain all the information needed to
evaluate the DNS equations in the required points. This is because $\sigma^h_{i j}$
is chosen such that the restricted DNS equations are closed in the discrete sense.

We use the divergence-form for the convective term, which is energy-conservative if
$\partial^h_j u^h_j = 0$~\cite{morinishiFullyConservativeHigher1998}.
The continuity equation $\partial^h_j u^h_j = 0$ is therefore enforced strictly
by using semi-explicit time discretization schemes
(applying a pressure projection \emph{after} each momentum time
step)~\cite{hairerNumericalSolutionDifferentialAlgebraic1989,
sanderseEnergyconservingRungeKutta2013}.

\subsection{Grid filters and the filter-swap property in 3D}

To obtain LES-FVM equations in conservative form, we need grid
filters that satisfy a filter-swap property.
As in the 1D case, the finite difference operator $\partial^h_i$ is associated
to a one-dimensional top-hat filter $g^h_i : U \to U$ 
that is only applied in the direction $e_i$. For all $u \in U$, we define it as
\begin{equation} \label{ns-gh}
    g^h_i u(x) \coloneq
    \frac{1}{h} \int_{-h / 2}^{h / 2} u(x + \alpha e_i) \, \mathrm{d} \alpha.
\end{equation}
We can use the fundamental theorem of calculus to show the
relation between $\partial^h_i$ and $g^h_i$
(with no summation over $i$):
\begin{equation} \label{eq:divergencetheorem}
    \partial^h_i = g^h_i \partial_i,
\end{equation}
meaning that the finite difference $\partial^h_i$
can be induced by filtering the exact derivative $\partial_i$.

The Navier-Stokes momentum and continuity equations include derivatives in
each of the cardinal directions $x_1$, $x_2$, and $x_3$. For example, the $i$-th
momentum equation includes the term
$\partial_j r_{i j} = \partial_1 r_{i 1} + \partial_2 r_{i 2} + \partial_3 r_{i 3}$.
If we filter the $i$-th momentum equation with the 1D grid filter $g^h_i$ in the
direction $i$, the equation for $g^h_i u_i$ would include the term
$g^h_i \partial_j r_{i j}$, and we could only use the filter-swap property
on one of the three terms, where $j = i$. A similar remark was made by
Lund, who argued that ideally, we would like
to filter each of the three derivatives $\partial_j$ with their associated
filters $g^h_j$ separately, but
such equations cannot be obtained by applying one single filter to
the Navier-Stokes momentum equations, since the same filter has to be applied to
all of the terms~\cite{lundUseExplicitFilters2003}.
We therefore resort to multi-dimensional grid filters.

From the 1D filters $g^h_i$, we define the multi-dimensional 
volume-averaging filter
\begin{equation} \label{eq:ns-fh}
f^h \coloneq g^h_1 g^h_2 g^h_3
\end{equation}
and surface-averaging filters
\begin{equation} \label{eq:ns-fhi}
    f^{h, 1} \coloneq g^h_2 g^h_3, \quad
    f^{h, 2} \coloneq g^h_1 g^h_3, \quad
    f^{h, 3} \coloneq g^h_1 g^h_2.
\end{equation}
For all $u \in U$, we employ the short-hand notation
\begin{equation}
    \bar{u}^{h} \coloneq f^h u, \quad
    \bar{u}^{h, i} \coloneq f^{h, i} u.
\end{equation}
A similar notation was used by Schumann~\cite{schumannSubgridScaleModel1975}.

By using the 1D property \eqref{eq:divergencetheorem} for $g^h_i$,
we obtain the following filter-swap commutation property for the
multi-dimensional grid filters
(with no sum over $i$):
\begin{equation} \label{eq:ns-filter-swap}
    f^h \partial_i = \partial^h_i f^{h, i}.
\end{equation}
This is a discrete equivalent of the ininitesimal property
$f^\Delta \partial_i = \partial_i f^\Delta$
which allows for switching between infinitesimal and
discrete derivatives. The important observation is that the filter definition
changes with the derivative definition.
The 1D case in \cref{eq:cl-continuous-commutation} can be seen as a special
case of \cref{eq:ns-filter-swap}, where the 1D volume is $[x \pm h / 2]$ and
the 1D surface is a single point in $\{ x \pm h / 2 \}$.

\subsection{LES-FVM}

The LES equation is $\overline{L(u)}^\Delta = 0$.
Applying the FVM filter $f^h$ gives the LES-FVM equation
\begin{equation}
    \overline{L(u)}^{\Delta, h} = 0.
\end{equation}
Using the filter-swap commutation property
\eqref{eq:ns-filter-swap}, we obtain the $\partial^h_j$-conservative form
\begin{equation} \label{eq:ns-u-Delta-h}
    \partial_t \bar{u}^{\Delta, h}_i +
    \partial^h_j \overline{r_{i j}(u)}^{\Delta, h, j} = 0.
\end{equation}
We can already note a couple of interesting observations.
The stress tensor $\overline{r_{i j}(u)}^{\Delta, h, j}$ (no sum over $j$)
is non-symmetric
(because it is averaged over the surface orthogonal to $j$, but not $i$)
and non-local in $u$ (due to the projection present inside $r$).
Furthermore, the LES-FVM solution
$\bar{u}^{\Delta, h}$ is not discretely divergence-free, since
$\partial^h_j \bar{u}^{\Delta, h}_j \neq
\overline{\partial_j u_j}^{\Delta, h} = 0$
(see \cref{th:noncommutation-coarsegraining}).
\footnote{
    An alternative is therefore to solve for the \emph{surface-averaged} field
    $\bar{u}^{\Delta, h, i}_i$ (no sum over $i$),
    since it becomes $\partial^h_j$-divergence-free after filtering
    ($\partial^h_j \bar{u}^{\Delta, h, j}_j = 0$
    )~\cite{kochkovMachineLearningAccelerated2021,
    agdesteinDiscretizeFirstFilter2025}.
    However, this field is not governed by a $\partial^h_j$-conservation law, since
    the surface-averaged momentum is not
    conserved discretely~\cite{agdesteinDiscretizeFirstFilter2025}.
    We will therefore not consider this option here.
}
Since the divergence-free constraint is important for stability in the
energy-conserving convective term,
we consider the divergence-free part of the volume-averaged field
$\bar{u}^{\Delta, h}$. This is precisely the projected field
\begin{equation}
    \bar{u}^{\Delta, h, \pi}_i \coloneq \pi^h_{i j} \bar{u}^{\Delta, h}_j.
\end{equation}
Applying the projector $\pi^h_{i j}$ to \cref{eq:ns-u-Delta-h}
and using \cref{th:projection-commutation} to swap projection and divergence
gives
\begin{equation} \label{eq:ns-u-Delta-h-pi}
    \partial_t \bar{u}^{\Delta, h, \pi}_i +
    \partial^h_j \pi^h_{i j \alpha \beta} \overline{r_{\alpha \beta}(u)}^{\Delta, h, \beta}
    = 0,
\end{equation}
which is a discrete conservation law for a divergence-free velocity field.

Adding the discrete flux term
$\partial^h_j r^h_{i j}(\bar{u}^{\Delta, h, \pi})$
to both sides gives the projected LES-FVM
equation in the form
\begin{equation} \label{eq:ns-les-fvm}
    L^h_i(\bar{u}^{\Delta, h, \pi}) = - \partial^h_j \tau^{\Delta, h, \pi}_{i j}(u),
\end{equation}
where the projected RST
$\tau^{\Delta, h, \pi}_{i j} \coloneq \pi^h_{i j \alpha \beta} \xi^{\Delta, h, \pi}_{\alpha \beta}$
is defined by the RST (with no sum over $j$)
\begin{equation} \label{eq:ns-xi-Delta-h-pi}
    \boxed{
    \xi^{\Delta, h, \pi}_{i j}(u) \coloneq
    \overline{r_{i j}(u)}^{\Delta, h, j} -
    \sigma^h_{i j}(\bar{u}^{\Delta, h, \pi}).
    }
\end{equation}
For comparison, the classical LES RST
for the projected double-filter $\pi f^h f^\Delta$ would be
$\tau^{\Delta, h, \pi}_{\text{Classic}, i j} \coloneq
\pi_{i j \alpha \beta} \xi^{\Delta, h, \pi}_{\text{Classic}, \alpha \beta}$,
where
\begin{equation}
    \xi^{\Delta, h, \pi}_{\text{Classic}, i j}(u) \coloneq
    \overline{\sigma_{i j}(u)}^{\Delta, h} -
    \sigma_{i j}(\bar{u}^{\Delta, h}).
\end{equation}
Note that $\bar{u}^{\Delta, h}$ is already infinitesimally divergence-free, so
$\pi_{i j} \bar{u}^{\Delta, h}_j = \bar{u}^{\Delta, h}_i$.

Like in the 1D case,
the discretization-informed RST $\xi^{\Delta, h, \pi}_{i j}(u)$ accounts for
both the filter-swap mechanism and the numerical stress $\sigma^h_{i j}$.
Unlike the 1D case, it also accounts for the pressure, since
the pressure that makes $\bar{u}^{\Delta, h, \pi}$ discretely
divergence-free is not equal to a filtered variant of the pressure that makes
$u$ infinitesimally divergence-free.

The LES-FVM RST $\xi^{\Delta, h, \pi}$ has two important differences from
the classical LES RST $\xi^{\Delta, h, \pi}_\text{Classic}$
(classical LES with $\pi f^h f^\Delta$ as the filter).
First, it is non-symmetric
($\xi^{\Delta, h, \pi}_{i j} \neq \xi^{\Delta, h, \pi}_{j i}$),
while the classical RST is symmetric
($\xi^{\Delta, h, \pi}_{\text{Classic}, i j} = \xi^{\Delta, h, \pi}_{\text{Classic}, j i}$).
Second, is it non-local in $u$ (due to the pressure projector),
while the classical RST is local in $u$.

\subsection{LES-FVM closure}

Let
$m^{\Delta, h, \pi}(\bar{u}^{\Delta, h, \pi}) \approx
\operatorname{dev} \left( \xi^{\Delta, h, \pi}(u) \right)$
be a grid-compatible LES-FVM closure model for the deviatoric part of the RST
defined as
$\operatorname{dev}(\sigma)_{i j} \coloneq
\sigma_{i j} - \delta_{i j} \sigma_{k k} / 3$
for all $\sigma \in U^{3 \times 3}$.
Since $\partial^h_j \pi^h_{i j \alpha \beta} p \delta_{\alpha \beta} = 0$
for all scalar fields $p \in U$,
we only need a closure model for the deviatoric part of the RST
(the isotropic part will be absorbed by the pressure projector).
The approximate LES-FVM equations are
\begin{equation} \label{eq:ns-les-fvm-closed}
    L^h_i(v^{\Delta, h, \pi}) = - \partial^h_j \pi^h_{i j \alpha \beta} m^{\Delta, h, \pi}_{\alpha \beta}(v^{\Delta, h, \pi}),
\end{equation}
where $v^{\Delta, h, \pi} \approx \bar{u}^{\Delta, h, \pi}$ is the approximate
LES-FVM solution.

We say that the closure $m^{\Delta, h, \pi}$ is structural if it is chosen to
approximate
$\xi^{\Delta, h, \pi}$ directly, and functional if it is chosen to produce the
same dissipation as $\xi^{\Delta, h, \pi}$.
In 1D, an RST and its dissipation coefficient are fields of the same dimension
(scalar field). The distinction between structural and functional models is
therefore less clear in 1D.
In 3D, the RST has $9$ components, while the dissipation coefficient has
$1$ component (scalar field). There is therefore more freedom in how to choose
a functional model than a structural model in 3D.

In the exact LES-FVM equation (for $\bar{u}^{\Delta, h, \pi}$),
the RST is non-symmetric and non-local in $u$.
This observation would suggest that the closure model in the approximate
LES-FVM equation (for $v^{\Delta, h, \pi}$)
should also be non-symmetric and non-local in $v^{\Delta, h, \pi}$.
In the exact LES equation (for $\bar{u}^\Delta$), the RST is both symmetric and
local in $u$, which justifies using a symmetric and local closure model
in classical LES (before discretization).
This means that a ``perfect'' LES-FVM closure model needs to predict $9$
different tensor components, while classical LES closures only need to predict
$6$ symmetric tensor components.

We now perform an experiment with the new RST expression.

\section{Experiment: DNS-aided LES-FVM for 3D turbulence}
\label{sec:navier-stokes-experiments}

Like for the Burgers equation, we employ a ``DNS-aided LES-FVM'' approach to
assess the importance of the RST definition in the LES-FVM equations.
We consider a 3D decaying turbulence test case in a periodic box.
The equations are defined by the following unitless parameters.
The domain size is $\ell \coloneq 1$.
The viscosity is $\nu \coloneq 2 \times 10^{-4}$.
The number of finite volumes in each of the 3 dimensions for DNS and LES are
$N_\text{DNS} \coloneq 500$ and $N \coloneq 100$
respectively, with a compression factor of $5$.
The grid spacings are $h_\text{DNS} \coloneq \ell / N_\text{DNS}$ and $h \coloneq \ell / N$
(note that these are called $h$ and $H$ in \ref{sec:two-grids}, respectively).
The LES filter $f^\Delta$ is a Gaussian of width $\Delta \coloneq 2 h$.
In the DNS discretization of the Gaussian kernel $G^\Delta$, we include points
up to $2$ standard deviations out on each side (see \ref{sec:two-grids} for details).
The infinitesimal expressions from \cref{sec:navier-stokes} are all approximated
on the DNS grid so that we can evaluate them based on the DNS solution,
but we will still use the infinitesimal notation to avoid visual clutter.

\subsection{Initialization}

Let $u \in U^3$ be a velocity field and $\kappa \in \mathbb{N}$ be a scalar
wavenumber.
We define the energy spectrum of $u$ at $\kappa$ as
\begin{equation} \label{eq:ns-spectrum}
    E(u, \kappa) \coloneq \frac{1}{2}
    \sum_{k \in K(\kappa)}
    \| \hat{u}(k) \|^2,
\end{equation}
where $\hat{u}(k)$ is the Fourier transform of $u$ at a wavenumber $k$ and 
$K(\kappa) \coloneq \{ k \in \mathbb{Z}^3 \ \mid
\ \kappa \leq \| k \| < \kappa + 1 \}$
is the shell of vector wavenumbers with magnitude
between $\kappa$ and $\kappa + 1$.

We initialize the solution $u$ on the DNS grid through the following
procedure, where $\leftarrow$ denotes the assignment operator.
\begin{enumerate}
    \item Sample a random field
    $u_i(x) \sim \mathcal{N}(0, 1)$
    from a normal distribution
    for each $i \in \{ 1, 2, 3 \}$ and
    each DNS grid point $x$.
    We do not define $u_i$ outside the grid points.
    \item Project the velocity: $u \leftarrow \pi u$.
    \item Compute the discrete Fourier transform
        $\hat{u} \leftarrow \operatorname{FFT}(u)$.
    \item For all wavenumbers
        $\kappa \in \{ 0, 1, \dots, \lfloor
        \sqrt{3} N_\text{DNS} / 2 \rfloor \}$,
        compute the current shell energy $E(u, \kappa)$,
        where $\lfloor \cdot \rfloor$
        denotes the integer part. For $\kappa \geq N / 2$,
        the shells are only partially filled,
        since the discrete Fourier transform gives a finite number of Fourier
        modes.
    Adjust the coefficients in the shell $K(\kappa)$ as
        \begin{equation}
            \hat{u}(k) \leftarrow
            \sqrt{\frac{P(\kappa)}{E(u, \kappa)}}
            \hat{u}(k), \quad \forall k \in K(\kappa),
        \end{equation}
        where $P(\kappa)$ is a prescribed energy profile defined as
        the Kolmogorov scaling in the inertial range as
        \begin{equation}
            P(\kappa) \coloneq \kappa^{-5 / 3}.
        \end{equation}
    \item Apply inverse Fourier transform
        $u \leftarrow \operatorname{IFFT}(\hat{u})$. 
    \item Reproject the velocity field (since the shell normalization may
        slightly perturb the DNS divergence of $u$):
        $u \leftarrow \pi u$.
    \item Scale the velocity field such that the total energy adds up to
        $1$:
        $u \leftarrow \frac{1}{\sqrt{1 / 2 \| u \|^2}} u$.
\end{enumerate}
The resulting velocity field $u$ is represented as an array of size
$N_\text{DNS}^3 \times 3$.
It is divergence-free,
the spectrum is proportional to the profile $P$
(with some deviations due to the second projection),
and the total energy is $1$.

Since the initial spectrum is artificial, we first run the DNS simulation for
$0.1$ time units to obtain a more realistic distribution of velocity scales.

\subsection{DNS-aided LES-FVM}

The DNS-aided LES-FVM formulation is defined as
\begin{align}
    L(u) = 0, \quad L^h(v^\text{model}) = - \nabla^h \cdot \tau^\text{model}(u),
\end{align}
where $u$ is the DNS solution, $v^\text{model}$ is the modeled
LES-FVM solution,
$\nabla^h \coloneq (\partial^h_1, \partial^h_2, \partial^h_3)$,
$\tau^\text{model} \coloneq \pi^h \xi^\text{model}$,
and $\xi^\text{model} : U^3 \to U^{3 \times 3}$ is an RST expression
taking the DNS solution as input
(to be defined in \cref{sec:RST-expressions}).
These equations are still continuous in time.
A forward-Euler time discretization with adaptive time stepping is
\begin{align}
    \Delta t^n
    & \coloneq 
    C \times \min \left(
        \frac{h_\text{DNS}}{\max \left| u^n \right|},
        \frac{h_\text{DNS}^2}{6 \nu}
    \right), \\
    u^{n + 1}
    & \coloneq u^n - \Delta t^n \, \nabla \cdot r(u^n), \\
    v^{n + 1}
    & \coloneq v^n - \Delta t^n \, \nabla^h \cdot
    \left(r^h(v^n) + \tau^\text{model}(u^n) \right),
\end{align}
where $\Delta t^n$ is chosen based on the DNS solution,
$u^n$ and $v^n$ are the Forward-Euler approximations to the
$u$ and $v^\text{model}$ at time
$t^n \coloneq \sum_{k = 0}^{n - 1} \Delta t^k$, and $C \coloneq 0.4$ is a
margin in the CFL expression.
The initial conditions are given by the final DNS field from the warm-up
simulation with $v^\text{model} = \bar{u}^{\Delta, h, \pi}$.
The goal is for $v^\text{model}$ to stay close to the projected filtered DNS
velocity field $\bar{u}^{\Delta, h, \pi}$ over time.

Since both $u = \pi u$ and $v^h = \pi^h v^h$ are divergence-free on
their respective grids, and \cref{th:projection-commutation}
can be used to commute projection and divergence, we can use the
``semi-explicit'' form of the scheme:
\begin{align}
    u^{n + 1}
    & = \pi \left[ u^n - \Delta t^n \nabla \cdot \sigma(u) \right], \\
    v^{n + 1}
    & = \pi^h \left[
        v^n - \Delta t^n \nabla^h \cdot
        \left(\sigma^h(v^n) + \xi^\text{model}(u^n) \right)
    \right], \label{eq:ns-lesfvm-forwardeuler}
\end{align}
which consists of doing an explicit forward-Euler step with the unprojected
stresses before reprojecting the velocity field back onto the space of
divergence-free fields.

Since turbulent flows are chaotic, we run the simulation for a short duration
of $0.1$ time units.
If we run for longer, the LES and DNS solutions will
decorrelate
and the DNS-aided closure term will be of little use
(unless $\xi^\text{model}$ is the new correct RST expression).

\subsection{RST expressions} \label{sec:RST-expressions}

We consider five DNS-aided ``closure'' models
defined (with no sum over $j$) by:
\begin{align}
    \xi^\text{A}_{i j}(u)
    & \coloneq 0, \\
    \xi^\text{B}_{i j}(u)
    & \coloneq
    \overline{\sigma_{i j}(u)}^{\Delta, h}  -
    \sigma_{i j}\left(\overline{u}^{\Delta, h}\right), \\
    \xi^\text{C}_{i j}(u)
    & \coloneq
    \overline{\sigma_{i j}(u)}^{\Delta, h}  -
    \sigma^h_{i j}\left(\overline{u}^{\Delta, h, \pi}\right), \\
    \xi^\text{D}_{i j}(u)
    & \coloneq
    \overline{\sigma_{i j}(u)}^{\Delta, h, j} -
    \sigma^h_{i j}\left(\overline{u}^{\Delta, h, \pi}\right), \\
    \xi^\text{E}_{i j}(u)
    & \coloneq
    \frac{1}{2} \left(
        \xi^\text{D}_{i j}(u) +
        \xi^\text{D}_{j i}(u)
    \right).
\end{align}
The respective LES-FVM solutions are denoted
$\{v^\text{A}, \dots, v^\text{E}\}$.
$\xi^\text{A}$ is the no-model case (DNS on the coarse grid).
$\xi^\text{B}$ is the classical RST expression for the given double-filter,
with all operators evaluated on the DNS grid.
Since $\pi f^{\Delta, h} = f^{\Delta, h} \pi$ and
$\pi u = u$,
the second term in $\xi^B_{i j}(u)$ simplifies from
$\sigma( \pi \bar{u}^{\Delta, h})$ to
$\sigma( \bar{u}^{\Delta, h})$.
$\xi^\text{C}$ accounts for the coarse grid discretization through the $h$-grid
numerical stress $\sigma^h$ in the second term.
Since the filtered field is not divergence-free on the $h$-grid
($\pi^h f^{\Delta, h} \neq f^{\Delta, H} \pi$),
the second term in $\xi^C$ includes the 
projector $\pi^h$. $\xi^\text{D}$ further accounts for the filter-swap
mechanism. The FVM filter is therefore surface-averaging in the first term and
not volume-averaging. $\xi^\text{E}$ is defined as the symmetric part of
$\xi^\text{D}$ (the surface-averaged term makes $\xi^\text{D}$
non-symmetric).

From \cref{eq:ns-xi-Delta-h-pi}, we know that $\xi^\text{D}$ is the correct
expression, since it accounts for the $h$-grid FVM discretization.
In an
``explicit LES'', which can be thought of as letting $h \to 0$, $\xi^\text{B}$
would be the correct expression. $\xi^\text{C}$ and $\xi^\text{E}$ are
included to assess the importance of the filter-swap mechanism and the
non-symmetric part of the correct RST $\xi^\text{D}$, respectively. As
classical LES closure models are designed to be symmetric, they are \emph{at
best} able to represent $\xi^\text{E}$, but not the non-symmetric RST
$\xi^\text{D}$.

\subsection{Results}

\begin{figure}
    \centering
    \includegraphics[width=\columnwidth]{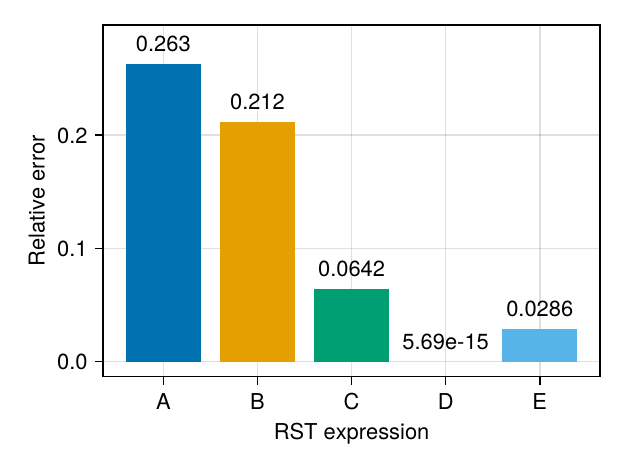}
    \caption{
        Errors from DNS-aided LES-FVM for the 3D decaying turbulence
        simulation.
    }
    \label{fig:ns-dnsaid-errors}
\end{figure}

\newcommand{\errA}{26.3 \%}
\newcommand{\errB}{21.2 \%}
\newcommand{\errC}{6.42 \%}
\newcommand{\errD}{0.00 \%}
\newcommand{\errE}{2.86 \%}

In \cref{fig:ns-dnsaid-errors}, we show the relative errors
\begin{equation}
    \frac{
        \| v^\text{model} - \bar{u}^{\Delta, h, \pi} \|_\Omega
    }{
        \| \bar{u}^{\Delta, h, \pi} \|_\Omega
    }
\end{equation}
at the final time $t = 0.1$ for the five RST expressions.
Our proposed RST expression $\xi^\text{D}$ gives an error at machine
precision, confirming that it is the correct RST expression for the LES-FVM
formulation.
The error for $\xi^\text{E}$ is at $\errE$, showing that the non-symmetric
part of the RST is indeed non-zero.
The error for $\xi^\text{C}$ is at $\errC$, a bit more than twice as large as $\xi^\text{E}$.
This shows that accounting for the filter-swap mechanism is also important.
The error for the classical expression $\xi^\text{B}$ is at $\errB$, a bit more than three
times larger than for $\xi^\text{C}$.
Finally, the no-model $\xi^\text{A}$ gives the largest error at $\errA$.
When comparing the errors for all the expressions, the
the largest absolute improvement is obtained by accounting for the numerical flux
(from $\xi^\text{B}$ to $\xi^\text{C}$).
All errors are evaluated after a short simulation of $0.1$ time units.
They will therefore continue to grow for longer simulations.

\begin{figure}
    \centering
    \includegraphics[width=\columnwidth]{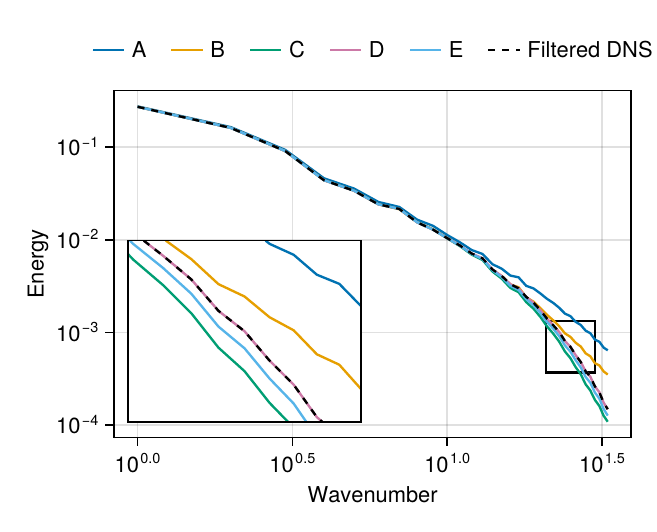}
    \caption{
        Energy spectra from DNS-aided LES-FVM for the 3D decaying turbulence test case.
    }
    \label{fig:ns-dnsaid-spectra}
\end{figure}

In \cref{fig:ns-dnsaid-spectra}, we show the energy spectra at the final time.
The spectrum of $\xi^\text{D}$ perfectly overlaps with the reference spectrum,
while the one of $\xi^\text{E}$ is lower, and the one of $\xi^\text{C}$ is
lower still. This means that not accounting for the non-symmetry and the
filter-swap property leads to a slight over-estimation of the energy
dissipation. On the other hand, the spectrum of the classical RST
$\xi^\text{B}$ is too high, meaning that not accounting for the discretization
in the RST leads to an under-estimation of the energy dissipation. Finally, the
no-model $\xi^\text{A}$ has the highest spectrum, showing that a dissipative
closure model is required to get the correct spectrum.

In the 1D Burgers' case (\cref{fig:burgers_spectrum}), many of the residual
flux expressions led to a sharp increase at the very end of the spectrum. This
was likely due to the shock-like structures that cannot be resolved on the
coarse grid. In \cref{fig:ns-dnsaid-spectra}, no such sharp increase at the end
of the spectrum is observed. Although the 3D resolution is smaller, which makes
it difficult to compare with the 1D case, the 3D experiment is incompressible,
and therefore cannot form the shock-like structures that were present in the
Burgers' case.

\section{Conclusion} \label{sec:conclusion}

In this work, we have proposed new exact expressions
for the residual stress tensor (RST) appearing in
discrete filtered conservation laws when using combined LES and FVM filters.
Unlike the classical RST expression $\overline{u u} - \bar{u} \bar{u}$
which is commonly studied in LES,
the RST in LES-FVM also includes contributions from the discrete divergence,
numerical flux, and discrete incompressibility constraint.
Writing the residual in discrete divergence form makes the RST non-symmetric,
while accounting for the discrete incompressibility constraint makes the RST
non-local.
To make these two properties more intuitive, we show how they are derived using
an alternative notation in \ref{sec:cartesian-notation}.

In a DNS-aided LES framework, where DNS data is used to compute the RST, our
RST expression gives errors at machine precision, while the classical RST
gives errors that grow over time, confirming the importance of accounting for
the discretization errors in the RST.

In practice, DNS-aided LES is only useful for evaluating the correctness of the
RST expression.
Does the new RST provide any practical benefits for LES closure modeling?
We believe so, and we have emphasized two main benefits:
\begin{enumerate}
    \item The new framework informs us about the structure of the RST, which
        can guide the development of new or existing closure models.
        Notably, the LES-FVM RST expression suggests that a ``perfect''
        structural closure model for incompressible flows should be
        non-symmetric and non-local in the velocity field. From the
        experiments in \cref{sec:navier-stokes-experiments}, we have some idea
        of the importance of the non-symmetric part, but the benefit of
        including non-local values in a closure model for incompressible flows
        is still unclear and requires further study.
    \item The new RST can be used as a target for tuning closure model
        parameters.
        By correctly accounting for the discretization in the RST,
        the RST automatically gives the correct desired dissipation profile
        that a functional closure model should reproduce. Notably, the correct
        RST accounts for the implicit dissipation produced by the numerical
        flux, so that the amount of explicitly modeled dissipation can be
        reduced accordingly. Our LES-FVM framework can thus be used as a tool
        to bridge the gap between implicit and explicit LES closure modeling.
\end{enumerate}

For the 1D Burgers equation, we demonstrated the benefit of the second point
by fitting the Smagorinsky model coefficient to the old and new RST
expression. Indeed, the coefficient fitted to the new RST was slightly higher,
adding some of the additional dissipation required by the FVM discretization
and leading to better results. This is despite the fact that the Smagorinsky
model only has one coefficient and is a rather inaccurate closure.

To compute the RST exactly from DNS reference data, we have introduced a new 
\emph{two-grid formulation} that accounts for two different overlapping
discretization levels (see \ref{sec:two-grids}).
This formulation is constructed such that the filter-swap commutation
properties hold discretely.
Making sure the DNS and LES grids are compatible in this way ensures that the
exact RST (with respect to the DNS reference) is computable.
Furthermore, as high-dimensional function approximators such as neural networks
are receiving more attention for LES closure
modeling~\cite{sanderseScientificMachineLearning2025}, access to accurate
discretization-informed target data becomes increasingly
important~\cite{beckDiscretizationConsistentClosureSchemes2023}.

For 1D conservation laws,
the central insight that leads to these discrete formulations was the partial
restoration of commutation between filtering and differentiation for finite
differences and coarsening grid filters. We showed that coarsening filters do
not commute with differentiation, but the partial commutation property of
finite differences was sufficient to obtain a discrete RST
expression. This commutation property also holds for higher-order finite
difference schemes, such as the ones studied by Geurts and van der
Bos~\cite{geurtsNumericallyInducedHighpass2005}. A similar formulation can be
derived for finite volume schemes on unstructured grids, since a
volume-averaged divergence over an unstructured grid cell can be written as a
surface integral over the volume faces.
Such an RST expression was proposed
by Denaro~\cite{denaroWhatDoesFinite2011,
denaroInconsistencefreeIntegralbasedDynamic2018}.

The key to making the LES-FVM framework valid for incompressible flows was to
rewrite the incompressible Navier-Stokes equations as a self-contained
conservation law for the velocity field. This was possible by incorporating the
incompressibility constraint and pressure gradient correction into a pressure
projection operator. This operator then became part of the stress tensor in the
self-contained conservation law for the velocity field.
In the presence of no-slip boundary conditions, this approach can still be used.
However, it is not yet clear to us if the pressure term can be written solely
in terms of the velocity for certain types of outflow boundary conditions that
prescribe a value for the pressure at the boundary.
This topic requires further study.

The LES-FVM framework we propose relies on first choosing a discretization
and a grid size, before choosing a closure model and tuning its parameters. We
therefore cannot expect the closure model to work well for a different
discretization or a different grid size. The discrete closure model must be
recalibrated to discretization-informed target data obtained from
DNS when the LES grid size is changed.

We demonstrated the LES-FVM framework for staggered grids, which are
composed of a primal and a dual grid (velocity and pressure points). The nodes
on the primal and dual grids do not overlap but are separated by half a grid
spacing. This led to the requirement that the primal and dual LES grids
overlap with the primal and dual DNS grids, respectively, by using
uniform Cartesian grids with odd compression factors. This is a limitation of
our framework on staggered grids. With some care, \emph{unstructured} staggered
grids can also be designed such that the primal and dual LES grids overlap with
the primal and dual DNS grids. This can be achieved by \emph{first} choosing
the LES grid and then refining it to obtain the DNS for generating training
data. For flows around objects (such as airfoils), the LES grid would therefore
need to be sufficiently fine to resolve the shape of the object considered. No
further refinement of the object boundary would be possible without losing the
exactness of the formulation. By allowing for some error at the boundary, this
requirement could be relaxed, while retaining the exact formulation in the
interior of the domain.

\section*{Software and reproducibility statement}

The code used to generate the results of this paper is available at
\url{https://github.com/agdestein/ExactClosure.jl}.
It is released under the MIT license.
An archived version is available at
\url{https://zenodo.org/records/16267799}.
The results in \cref{sec:burgers} were obtained on a desktop CPU.
The results in \cref{sec:navier-stokes-experiments} were obtained
using a single Nvidia H100 GPU on the Dutch National Supercomputer Snellius.

To allow for reproducibility on accessible hardware,
the code also includes a scaled down version of the 3D experiment
($90^3$ DNS grid points and $\{ 18^3, 30^3 \}$ LES grid points)
that can be run with multithreading on a modern laptop CPU.
The results for the scaled-down 3D experiment are
similar to those of the full experiment.

\section*{CRediT author statement}

\textbf{Syver Døving Agdestein}:
Conceptualization,
Methodology,
Software,
Validation,
Visualization,
Writing -- Original Draft

\textbf{Roel Verstappen}:
Supervision,
Writing -- Review \& Editing

\textbf{Benjamin Sanderse}:
Formal analysis,
Funding acquisition,
Project administration,
Supervision,
Writing -- Review \& Editing

\section*{Declaration of Generative AI and AI-assisted technologies in the writing process}

During the preparation of this work the authors used GitHub Copilot in order to
propose wordings and mathematical typesetting. After using this tool/service,
the authors reviewed and edited the content as needed and take full
responsibility for the content of the publication.

\section*{Declaration of competing interest}

The authors declare that they have no known competing financial interests or
personal relationships that could have appeared to influence the work reported
in this paper.

\section*{Acknowledgements}

This work is supported by the projects ``Discretize first, reduce next'' (with
project number VI.Vidi.193.105) of the research programme NWO Talent Programme
Vidi and ``Discovering deep physics models with differentiable programming''
(with project number EINF-8705) of the Dutch Collaborating University Computing
Facilities (SURF), both financed by the Dutch Research Council (NWO). We thank
SURF (www.surf.nl) for the support in using the Dutch National Supercomputer
Snellius.

\appendix

\section{Grid-compatible operators and restriction} \label{sec:grid-compatibility}

We say that an operator $o : U \to U$ is $h$-collocated if,
for all $x \in \Omega$, $o(u)(x)$ only depends on the values
$\{u(x + i h) \mid i \in \mathbb{Z} \}$.
Similarly, $o$ is $h$-staggered if $o(u)(x)$ only depends on
$\{u(x + i h + h / 2) \mid i \in \mathbb{Z} \}$.

The finite difference $\partial^h_x$, interpolator $\eta^h_x$, and numerical Burgers flux $r^h$ from
\cref{eq:cl-finite-difference,eq:cl-interpolator,eq:burgers-discrete}
are staggered since $\partial^h_x u(x)$, $\eta^h_x u(x)$, and $r^h(u)(x)$
depend on $u(x \pm h / 2)$, and not on $u(x \pm h)$. When chained together, the
staggered operators become collocated. For example, the expressions
\begin{align}
    \partial^h_x \partial^h_x u(x)
    & = \frac{u(x + h) - 2 u(x) + u(x - h)}{h^2}, \\
    \partial^h_x \eta^h_x u(x)
    & = \frac{u(x + h) - u(x - h)}{2h},
\end{align}
only depend on $u(x - h)$, $u(x)$, and $u(x + h)$.
These expressions do not depend on $u$ at the half-points $x \pm h / 2$.

We say that an operator
(such as the discrete Laplacian $\partial^h_x \partial^h_x$)
is \emph{compatible} with a grid of spacing $h$ if the operator is collocated.
This means that equations involving the operator form a closed
system of equations once restricted to the grid points.
In particular, a flux term of the form $\partial^h_x r^h$ is grid-compatible
if $r^h : U \to U$ is a staggered numerical flux.
We therefore say that a flux is grid-compatible if it is staggered.

At a point $x$, the numerical flux Burgers flux 
$r^h(u) \coloneq (\eta^h_x u)(\eta^h_x u) / 2 - \nu \partial^h_x u$ takes the form
\begin{equation}
\label{eq:flux-in-a-point}
\begin{split}
    r^h(u)(x)
    & = \frac{1}{8}
    \left[ u\left(x - \frac{h}{2}\right) + u\left(x + \frac{h}{2}\right) \right]^2 \\
    & - \frac{\nu}{h} \left[ u\left(x + \frac{h}{2}\right) - u\left(x - \frac{h}{2}\right) \right],
\end{split}
\end{equation}
and the discrete flux divergence is
\begin{equation}
    \label{eq:flux-term-in-a-point}
    \begin{split}
        \partial^h_x r^h(u)(x)
        & = \frac{1}{h}
        \left[ r^h(u)\left(x + \frac{h}{2}\right) - r^h(u)\left(x - \frac{h}{2}\right) \right] \\
        & = \frac{1}{8 h} \left[ u(x + h)^2 - u(x - h)^2 \right] \\
        & + \frac{1}{8 h} u(x) \left[ u(x + h) - u(x - h) \right] \\
        & - \frac{\nu}{h^2} \left[ u(x + h) - 2 u(x) + u(x - h) \right].
    \end{split}
\end{equation}
This expression can be evaluated at any point $x \in \Omega$.
If we evaluate the continuous field $\partial^h_x r^h(u)$ at the
grid points $x^h_i \coloneq i h$, $i \in \mathbb{Z}$, and define
the restriction $u^h_i \coloneq u(x^h_i)$,
we can use the more common (but less general) notation style 
\begin{equation}
    \label{eq:flux-term-in-a-grid-point-restricted}
    \begin{split}
        \partial^h_x r^h(u)(x^h_i)
        & = \frac{1}{8 h} \left[
            (u^h_{i + 1})^2 - (u^h_{i - 1})^2 +
            u^h_i \left( u^h_{i + 1} - u^h_{i - 1} \right)
        \right] \\
        & - \frac{\nu}{h^2} \left[ u^h_{i + 1} - 2 u^h_i + u^h_{i - 1} \right].
    \end{split}
\end{equation}
This restricted form, where the expression is written only in terms of
$(u^h_i)_{i \in \mathbb{Z}}$, is only possible to write because
$\partial^h_x r^h$ is a grid-compatible operator.
We still use the continuous notation like in
\cref{eq:burgers-discrete,eq:flux-in-a-point,eq:flux-term-in-a-point}
for analysis
and only use the restricted
form like \eqref{eq:flux-term-in-a-grid-point-restricted} for computer
evaluation.

\section{Grid-compatible coarse-graining}
\label{sec:two-grids}

\begin{figure*}
    \centering
    \def\svgwidth{\textwidth}
    \includegraphics[width=\textwidth]{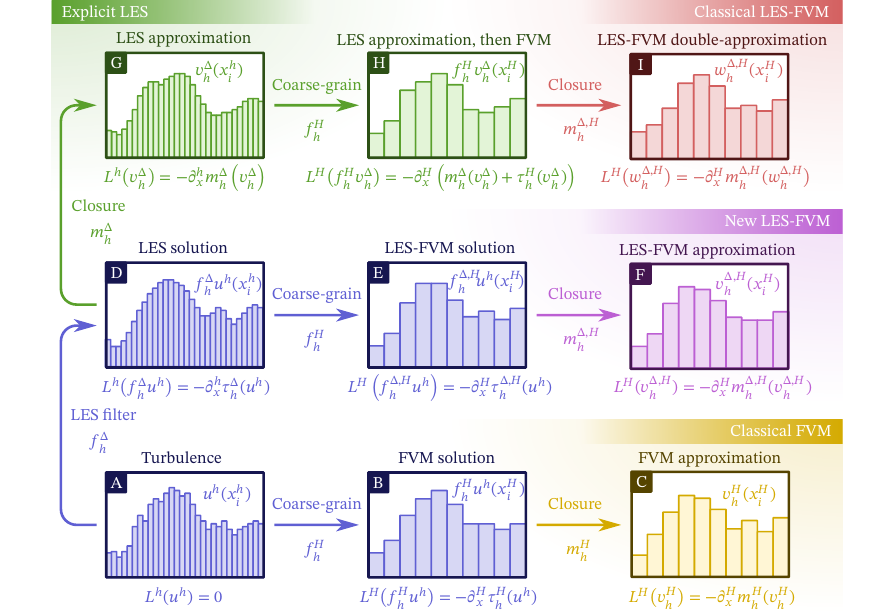}
    \caption{
        Same as \cref{fig:threepaths}, but starting from the DNS solution $u^h$
        instead of the continuous solution $u$.
        In the limit when $h$ goes to zero, this figure becomes equivalent to
        \cref{fig:threepaths}.
        The fine grid points are $(x^h_i)_{i \in \mathbb{Z}}$,
        and the coarse grid points are $(x^H_i)_{i \in \mathbb{Z}}$.
    }
    \label{fig:threepaths-discrete}
\end{figure*}

A problem with the LES-FVM framework
is that we cannot compute target solutions $\bar{u}^{\Delta, h}$ and target
residual fluxes $\tau^{\Delta, h}(u)$ unless we have
access to a continuous solution $u$ in all points $x \in \Omega$.
In practice, reference data is obtained by first doing DNS
by solving
the DNS equation
\begin{equation} \label{eq:dns}
    L^h(u^h) = 0
\end{equation}
with a sufficiently small DNS grid spacing $h$.
This gives the DNS solution $u^h$ at
the $h$-grid points $x^h_i \coloneq i h$, $i \in \mathbb{Z}$.
The LES-FVM problem is then formulated on a coarser
grid with spacing $H > h$.
The goal is to approximate the fields
$\bar{u}^{\Delta, H}(x^H_i)$ and $\tau^{\Delta, H}(u)(x^H_i)$
evaluated at the $H$-grid points $(x^H_i)_{i \in \mathbb{Z}}$
using only the DNS solution $u^h(x^h_i)$ evaluated at the $h$-grid points
$(x^h_i)_{i \in \mathbb{Z}}$.
This framework is illustrated in \cref{fig:threepaths-discrete}.
This two-grid formulation should approximate the one-grid formulation
and converge to it as $h$ goes to $0$.

\subsection{1D filters}

For the 1D case,
we propose the following DNS approximations of the filters that fits our
criteria.
Choose $H \coloneq (2 n + 1) h$ for some $n \in \mathbb{N}$.
The FVM filter $f^H$ is replaced by the
$h$-compatible filter $f^H_h : U \to U$
using the quadrature rule for all $u \in U$ and $x \in \Omega$ as
\begin{equation} \label{eq:cl-two-grid-filter}
    f^H_h u(x) \coloneq \frac{1}{2 n + 1} \sum_{i = -n}^n u(x + i h).
\end{equation}
This filter is designed to satisfy
a discrete equivalent of the filter-swap commutation property
\eqref{eq:cl-continuous-commutation}:
\begin{equation} \label{eq:cl-discrete-commutation}
    \partial^H_x = f^H_h \partial^h_x.
\end{equation}
For proof, see \cref{th:commutation-fHh}.

Assume that the LES kernel $G^\Delta$ of the LES filter $f^\Delta$ is compactly
supported (e.g.\ top-hat) or decays sufficiently fast at infinity
(e.g.\ Gaussian).
We then replace $f^\Delta$ with the
$h$-compatible filter $f^\Delta_h : U \to U$ defined as
\begin{equation}
    f^\Delta_h u(x)
    = \frac{
        \sum_{r = -R}^R G^\Delta(r h) u(x - r h)
    }{
        \sum_{r = -R}^R G^\Delta(r h)
    }
\end{equation}
for some sufficiently large $R \in \mathbb{N}$.
The denominator ensures that the discrete kernel is normalized.
For the Gaussian filter \eqref{eq:gaussian}, the standard deviation is
$\sigma = \Delta / \sqrt{12}$. We then choose $R$ such that
$R h \geq 3 \sigma$, i.e.\
$R \coloneq \lceil 3 \Delta / (\sqrt{12} h) \rceil$, where
$\lceil \cdot \rceil$ is the ceiling function.
The filter $f^\Delta_h$ is $h$-compatible since
$f^\Delta_h u(x)$ only depend on $u$ at $x$ modulo $h$.

Applying the coarse-graining $h$-compatible LES-FVM filter
$f^{\Delta, H}_h \coloneq f^H_h f^\Delta_h$
to the DNS equation \eqref{eq:dns} gives the LES-FVM equation
\begin{equation}
    f^{\Delta, H}_h L^h(u^h) = 0.
\end{equation}
By adding the resolved term $\partial^H_x r^H(f^{\Delta, H}_h u^h)$
to both sides and using the discrete filter-swap property
\eqref{eq:cl-discrete-commutation},
we get the LES-FVM equation in $H$-grid conservative form:
\begin{equation} \label{eq:cl-les-fvm-twogrid}
    L^H(f^{\Delta, H}_h u^h) = - \partial^H_x \tau^{\Delta, H}_h(u^h)
\end{equation}
(see \cref{fig:threepaths-discrete}-E),
where $f^{\Delta, H}_h u^h$ is the coarse-grained LES-FVM solution
we intend to solve for and
\begin{equation}
    \boxed{
        \tau^{\Delta, H}_h(u^h) \coloneq
        f^\Delta_h r^h(u^h) - r^H\left(f^{\Delta, H}_h u^h\right)
    }
\end{equation}
is the residual flux in the LES-FVM equation.
\Cref{eq:cl-les-fvm-twogrid} is analogous to the LES-FVM equation
\eqref{eq:cl-les-fvm}, but it is
written using $H$-grid divergences $\partial^H_x$.

Note that the FVM filters $f^h$, $f^H$, and $f^H_h$ are related through the
property
\begin{equation}
    f^H = f^H_h f^h.
\end{equation}
For proof, see \cref{th:filterlevels}.
At fixed $H$, we can show that $f^H_h \to f^H$ and $f^\Delta_h \to f^\Delta$
as $h \to 0$.
This means that if the DNS is fully resolved, we recover the continuous
setting.

A limitation of our FVM filter $f^H_h$ is that
the compression factor is required to be odd, i.e.\
$H = (2 n + 1) h$ for some $n$, and not $H = 2 n h$.
For an odd compression factor, we can compute both
$f^{\Delta, H}_h u^h$ and and $\tau^{\Delta, H}_h(u^h)$ in the required
staggered grid points exactly, without performing any interpolations.
This would not be possible for an even compression factor, due to the way the
staggered coarse and fine grids overlap.

\subsection{Coarse-graining filters in 3D}

The 1D $h$-compatible LES and FVM filters can easily be extended to 3D
by applying them successively in each coordinate direction.
The 1D FVM filter 
$g^H_{h, i} : U \to U$ can be defined as
\begin{equation}
    g^H_{h, i} u(x)
    \coloneq \frac{1}{2 n + 1} \sum_{r = -n}^n u(x + r h e_i),
\end{equation}
where $e_i$ is the $i$-th unit vector.
The volume-averaging and surface-averaging FVM filters then follow as
$f^H_h \coloneq g^H_{h, 1} g^H_{h, 2} g^H_{h, 3}$,
$f^{H, 1}_h \coloneq g^H_{h, 2} g^H_{h, 3}$, 
$f^{H, 2}_h \coloneq g^H_{h, 1} g^H_{h, 3}$, and
$f^{H, 3}_h \coloneq g^H_{h, 1} g^H_{h, 2}$.

Like their one-grid counterparts, our proposed two-grid filters have the
commutation properties (with no sum over $i$)
\begin{align}
    g^H_{h, i} \partial^h_i & = \partial^H_i, \\
    f^H_h \partial^h_i & = \partial^H_i f^{H, i}_h.
\end{align}

Given a DNS solution $u^h$ restricted to the $h$-grid, target data pairs
$f^{\Delta, H, \pi}_h u^h$ and $\tau^{\Delta, H, \pi}_h(u^h)$ can be computed and then
evaluated in the required $H$-grid points to assess the performance of LES-FVM
closures.

\section{Proofs of commutation properties} \label{sec:proofs}

Here we provide proofs for various properties used in this article.
We recall that $\Omega$ is a periodic 1D domain,
$U$ is the space of periodic 1D fields on $\Omega$,
$f^\Delta$ is a spatial convolutional LES filter (see \cref{eq:convolution}),
$f^h$ is an FVM filter (see \cref{eq:cl-gridfilter}),
$f^\Delta_h$ is a $h$-grid-compatible LES filter
(see \cref{eq:cl-two-grid-filter}),
$f^H_h$ is a $h$-grid-compatible FVM filter (see \cref{eq:cl-two-grid-filter}),
$\partial^h_x$ is a finite difference operator (see \cref{eq:cl-finite-difference}),
$h$ is a grid spacing, and
$H = (2 n + 1) h$ is a coarse grid spacing for some $n \in \mathbb{N}$.

Note that for non-uniform filters or bounded domains,
some of the commutation properties may no longer hold.
Here, we only consider periodic domains.

\begin{theorem} \label{th:commutation-f-partial}
    Spatial convolutional filters $f^\Delta$ commute with
    differentiation~\cite{berselliMathematicsLargeEddy2006}:
    \begin{equation}
        f^\Delta \partial_x = \partial_x f^\Delta.
    \end{equation}
\end{theorem}

\begin{proof}
    Let $u \in U$ and $x \in \Omega$. Then
    \begin{equation}
    \begin{split}
        \partial_x \bar{u}^\Delta(x)
        & = \partial_x \left[ \int_\mathbb{R} G^\Delta(x - y) u(y) \, \mathrm{d} y \right] \\
        & = \int_\mathbb{R} \partial_x [G^\Delta(x - y)] u(y) \, \mathrm{d} y \\
        & = \int_\mathbb{R} (\partial_x G^\Delta)(x - y) u(y) \, \mathrm{d} y \\
        & = -\int_\mathbb{R} \partial_y [G^\Delta(x - y)] u(y) \, \mathrm{d} y \\
        & = \int_\mathbb{R} G^\Delta(x - y) \partial_y u(y) \, \mathrm{d} y \\
        & = \overline{\partial_x u}^\Delta(x).
    \end{split}
    \end{equation}
    where we first used Leiblnitz's rule to interchange differentiation and
    integration, then the chain rule, and then integration by parts (assuming that
    the kernel goes to zero at infinity).
    Since this holds for all $u$ and $x$, we have $f \partial_x = \partial_x f$.
\end{proof}

\begin{theorem} \label{th:noncommutation-coarsegraining}
    Coarse-graining and differentiation do not commute:
    \begin{equation}
        \partial^h_x f^h \neq f^h \partial_x.
    \end{equation}
\end{theorem}

\begin{proof}
    Let $u \in U$.
    The Taylor series expansion of $u$ around a point $x$ is
    \begin{equation}
    \begin{split}
        u\left(x + \frac{h}{2}\right)
        & = u(x)
        + \frac{h}{2} \partial_x u(x)
        + \frac{h^2}{8} \partial_{x x} u(x) \\
        & + \frac{h^3}{48} \partial_{x x x} u(x)
        + \frac{h^4}{384} \partial_{x x x x} u(x)
        + \mathcal{O}(h^5),
    \end{split}
    \end{equation}
    where $\partial_{x x} \coloneq \partial_x \partial_x$ etc.
    Subtracting a similar expansion of $u(x - h / 2)$ makes the even terms
    cancel out. The expansion of the finite difference operator $\partial^h_x$ is
    therefore reduced to
    \begin{equation}
        \partial^h_x = \partial_x - \frac{h^2}{24} \partial_{x x x} + \mathcal{O}(h^4).
    \end{equation}
    This gives
    \begin{equation}
        \begin{split}
            \partial^h_x f^h
            & = \partial_x f^h - \frac{h^2}{24} \partial_{x x x} f^h + \mathcal{O}(h^4) \\
            & = f^h \partial_x - \frac{h^2}{24} f^h \partial_{x x x} + \mathcal{O}(h^4) \\
            & \neq f^h \partial_x,
        \end{split}
    \end{equation}
    since in the operator $f^h \partial_{x x x}$ is non-zero. Here we used
    \cref{th:commutation-f-partial} to swap $f^h$ and $\partial_x$.
    Note that for a few special cases,
    such as velocity fields with $\partial_{x x x} u = 0$
    (and similarly for higher order derivatives),
    we do get $\partial^h_x \bar{u}^h = \overline{\partial_x u}^h$.
\end{proof}

\begin{theorem} \label{th:commutation-fh}
    The finite difference $\partial^h_x$ can be written as a composition between
    the FVM filter and an exact derivative:
    \begin{equation}
        \partial^h_x = f^h \partial_x.
    \end{equation}
\end{theorem}

\begin{proof}
    The fundamental theorem of calculus states that
    \begin{equation}
        \int_a^b \partial_x u \, \mathrm{d} x = u(b) - u(a)
    \end{equation}
    for all $(a, b) \in \mathbb{R}^2$.
    For $u \in U$ and $x \in \Omega$, this gives
    \begin{equation}
        \begin{split}
            \overline{\partial_x u}^h(x)
            & = \frac{1}{h} \int_{x -h / 2}^{x + h / 2} \partial_y u(y) \, \mathrm{d} y \\
            & = \frac{1}{h} \left[ u\left(x + \frac{h}{2}\right) - u\left(x - \frac{h}{2}\right) \right] \\
            & = \partial^h_x u(x).
        \end{split}
    \end{equation}
    Since this holds for all $u$ and $x$, we have $\partial^h_x = f^h \partial_x$.
\end{proof}

We now show the properties of the coarse-graining filter $f^H_h$.

\begin{theorem} \label{th:filterlevels}
    The average over $H \coloneq (2 n + 1) h$ can be written as a composition
    between
    a coarse-graining filter and the average over $h$:
    \begin{equation}
        f^H = f^H_h f^h,
    \end{equation}
    where $n \in \mathbb{N}$.
\end{theorem}

\begin{proof}
    Let $u \in U$ and $x \in \Omega$. Then
    \begin{equation}
    \begin{split}
        f^H_h \bar{u}^h(x)
        & = \frac{1}{2 n + 1}
        \sum_{i = -n}^n
        \frac{1}{h}
        \int_{x + i h - h/2}^{x + i h + h/2}
        u(y) \, \mathrm{d} y \\
        & = \frac{1}{(2 n + 1) h} \int_{x - (2 n + 1) h / 2}^{x + (2 n + 1) h / 2} u(y) \, \mathrm{d} y \\
        & = \frac{1}{H} \int_{x - H / 2}^{x + H / 2} u(y) \, \mathrm{d} y \\
        & = \bar{u}^H(x),
    \end{split}
    \end{equation}
    where we used the property
    $\int_a^b u(x) \mathrm{d} x
    + \int_b^c u(x) \mathrm{d} x
    = \int_a^c u(x) \mathrm{d} x$
    for all $(a, b, c) \in \mathbb{R}^3$
    to combine the integrals in the sum.
    Since this holds for all $u$ and $x$, we have $f^H = f^{h \to H} f^h$.
\end{proof}

\begin{theorem} \label{th:commutation-fHh}
    A finite difference over $H \coloneq (2 n + 1) h$
    can be written as a composition between
    a coarse-graining filter and a finite difference over $h$:
    \begin{equation}
        \partial^H_x = f^H_h \partial^h_x,
    \end{equation}
    where $n \in \mathbb{N}$.
\end{theorem}

\begin{proof}
    Using \cref{th:commutation-fh,th:filterlevels}, we have
    \begin{equation}
    \partial^H_x = f^H \partial_x = f^H_h f^h \partial_x = f^H_h \partial^h_x.
    \end{equation}
\end{proof}

\begin{figure}
    \centering
    \includegraphics[width=\columnwidth]{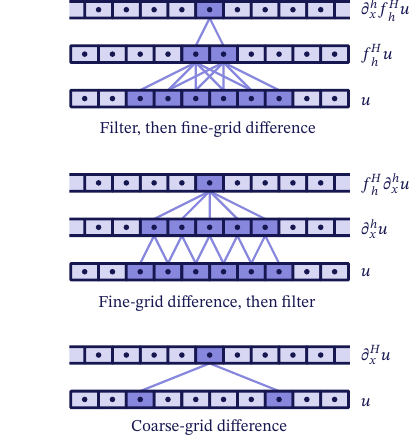}
    \caption{
        For all $u \in U$,
        the three terms 
        $\partial^H_x u$,
        $f^H_h \partial^h_x u$, and
        $\partial^h_x f^H_h u$
        are equal.
        Here we show the three quantities on the
        $h$-grid with $H \coloneq 5 h$.
    }
    \label{fig:commutation}
\end{figure}

Note that \cref{th:commutation-fHh} can be extended to include three terms:
$\partial^H_x = \partial^h_x f^H_h = f^H_h \partial^h_x$, since filtering
and finite differencing do commute
\emph{if we do not coarse-grain the derivative}.
These three terms are shown in \cref{fig:commutation} for $H = 5 h$.
The three terms are equal since all the intermediate terms cancel out in the
telescoping sum in $f^H_h$.
We restrict the fields to the $h$-grid to visualize how the inner terms cancel
out.

\section{Pressure projection for vectors and stress tensors} \label{sec:projection}

The system \eqref{eq:navier-stokes} consist of an evolution equation subject to
a spatial constraint. By defining a pressure projection operator,
these equations can be combined into one self-contained evolution equation.

In this appendix, we introduce two projection operators, one that makes vector
fields in $U^3$ divergence-free and one that makes stress tensor fields in
$U^{3 \times 3}$ divergence-preserving.
We provide proofs for the continuous case. For the discrete case, the proofs are
the identical, the only difference is that instead of $\partial_i$ we use $\partial^h_i$.

\subsection{Pressure projector for vector fields}

The pressure field $p$ enforces the continuity equation $\partial_j u_j = 0$.
In \cref{eq:navier-stokes},
combining the continuity equation with the momentum equation gives
the Poisson equation for the pressure field:
\begin{equation} \label{eq:ns-poisson}
    -\partial_k \partial_k p = \partial_i \partial_j \sigma_{i j}(u).
\end{equation}
By solving this equation explicitly for the pressure, we can write a 
``pressure-free'' momentum equation as
\begin{equation} \label{eq:ns-pi}
    \partial_t u_i + \pi_{i j} \partial_k \sigma_{j k}(u) = 0,
\end{equation}
where
\begin{equation}
    \pi_{i j} \coloneq \delta_{i j} - \partial_{i} (\partial_k \partial_k)^{\dagger} \partial_j
\end{equation}
is a pressure projection operator (see \cref{th:vector-projector} for proof),
$\delta_{i j}$ is the Kronecker symbol,
and the inverse Laplacian $(\partial_k \partial_k)^{\dagger} : \varphi \mapsto p$
maps scalar fields $\varphi$ to the unique solution to the
Poisson equation $\partial_k \partial_k p = \varphi$
subject to the additional constraint of an average pressure of zero,
i.e.\ $\int_\Omega p \, \mathrm{d} V = 0$.
For our periodic domain, the pressure field is determined up to a constant.
We are free to choose the constant this way since it subsequently disappears in the
pressure gradient $\partial_i p$.

\begin{theorem} \label{th:vector-projector}
    The operator
    $\pi_{i j} \coloneq \delta_{i j} - \partial_i (\partial_k \partial_k)^{\dagger} \partial_j$
    is a projector onto the space of divergence-free vector fields,
    i.e.\ $\partial_i \pi_{i j} = 0$ (the output of $\pi$ is divergence-free)
    and $\pi \pi = \pi$
    \cite{chorinNumericalSolutionNavierStokes1968,
    temamLapproximationSolutionEquations1969,
    vremanProjectionMethodIncompressible2014}.
\end{theorem}

\begin{proof}
    We have
    $\partial_i \partial_i (\partial_j \partial_j)^{\dagger} = 1$,
    since $(\partial_j \partial_j)^{\dagger}$ gives solutions to the Poisson equation.
    This can be used to show that $\pi_{i j}$ makes vector fields divergence-free.
    The divergence $\partial_i$ composed with $\pi_{i j}$ is
    \begin{equation}
        \partial_i \pi_{i j}
        = \underbrace{\partial_i \delta_{i j}}_{\partial_j}
        - \underbrace{\partial_i \partial_i (\partial_k \partial_k)^{\dagger}}_{1} \partial_j
        = \partial_j - \partial_j
        = 0.
    \end{equation}
    This means that for all $u \in U^3$,
    $\partial_i \pi_{i j} u_j = 0$,
    so $\pi u$ is divergence-free
    (even if $u$ is not).
    Additionally, we get
    \begin{equation}
        \begin{split}
            \pi_{i j} \pi_{j k}
            & = \delta_{i j} \delta_{j k} \\
            & - \delta_{i j} \partial_{j} (\partial_\alpha \partial_\alpha)^{\dagger} \partial_k
            - \delta_{j k} \partial_{i} (\partial_\alpha \partial_\alpha)^{\dagger} \partial_j \\
            & + \partial_{i}
            (\partial_\alpha \partial_\alpha)^{\dagger}
            \underbrace{\partial_j \partial_{j} (\partial_\beta \partial_\beta)^{\dagger}}_{1}
            \partial_k \\
            & = \delta_{i k} - 2 \partial_{i} (\partial_\alpha \partial_\alpha)^{\dagger}
            \partial_k
            + \partial_{i} (\partial_\alpha \partial_\alpha)^{\dagger} \partial_k \\
            & = \pi_{i k}.
        \end{split}
    \end{equation}
    Since $\pi \pi = \pi$, we can conclude that $\pi$ is idempotent.
\end{proof}

The projected momentum equation \eqref{eq:ns-pi}
automatically enforces the continuity equation at all times
\emph{by construction}
(as long as $\partial_j u_j = 0$ at the initial time).
The pressure term and continuity equation can be ignored,
at the cost of making the momentum equations non-local
(the inverse Laplacian is non-local).

\subsection{Pressure projector for tensor fields}

We say that a stress tensor $\sigma \in U^{3 \times 3}$
is divergence-preserving if
$\partial_j \sigma_{i j}$ is divergence-free, i.e.\
$\partial_i \partial_j \sigma_{i j} = 0$.
Define the tensor projector $\pi : U^{3 \times 3} \to U^{3 \times 3}$ as
\begin{equation}
    \pi_{i j \alpha \beta} \coloneq \delta_{i \alpha} \delta_{j \beta} -
    \delta_{i j} \left( \partial_k \partial_k \right)^{\dagger}
    \partial_\alpha \partial_\beta.
\end{equation}
This operator maps stress tensors to stress tensors.

\begin{theorem} \label{th:stress-projector}
    The operator $\pi_{i j \alpha \beta}$ is a projector onto the space of
    divergence-preserving stress tensors, i.e.\
    $\partial_i \partial_j \pi_{i j \alpha \beta} = 0$
    (the output of $\pi$ is divergence-preserving) and
    $\pi \pi = \pi$.
\end{theorem}

\begin{proof}
    The double-divergence $\partial_i \partial_j$ composed with the operator
    $\pi_{i j \alpha \beta}$ is
    \begin{equation}
    \begin{split}
        \partial_i \partial_j \pi_{i j \alpha \beta}
        & = \delta_{i \alpha} \delta_{j \beta} \partial_i \partial_j
        - \delta_{i j} \partial_i \partial_j
        \left( \partial_k \partial_k \right)^{\dagger}
        \partial_\alpha \partial_\beta \\
        & = \partial_\alpha \partial_\beta
        - \underbrace{\partial_i \partial_i
        \left( \partial_k \partial_k \right)^{\dagger}}_{1}
        \partial_\alpha \partial_\beta \\
        & = \partial_\alpha \partial_\beta - \partial_\alpha \partial_\beta \\
        & = 0.
    \end{split}
    \end{equation}
    This means that for all $\sigma \in U^{3 \times 3}$, we have
    $\partial_i \partial_j ( \pi_{i j \alpha \beta} \sigma_{\alpha \beta} ) = 0$,
    so $\pi \sigma$ is a divergence-preserving
    tensor.

    Furthermore, applying the operator twice gives
    \begin{equation}
    \begin{split}
        \pi_{i j \alpha \beta} \pi_{\alpha \beta m n} =
        & \left( \delta_{i \alpha} \delta_{j \beta} - \delta_{i j} (\partial_k \partial_k)^{\dagger}
        \partial_\alpha \partial_\beta \right) \\
        & \left( \delta_{\alpha m} \delta_{\beta n} - \delta_{\alpha \beta} (\partial_l \partial_l)^{\dagger} \partial_m \partial_n \right) \\
        = &
        (\delta_{i \alpha} \delta_{j \beta}) (\delta_{\alpha m} \delta_{\beta n}) \\
        - &
        (\delta_{i \alpha} \delta_{j \beta}) \delta_{\alpha \beta} (\partial_l \partial_l)^{\dagger} \partial_m \partial_n \\
        - &
        (\delta_{\alpha m} \delta_{\beta n}) \delta_{i j} (\partial_k \partial_k)^{\dagger} \partial_\alpha \partial_\beta \\
        + &
        \delta_{i j} \delta_{\alpha \beta}
        (\partial_k \partial_k)^{\dagger} \partial_\alpha \partial_\beta
        (\partial_l \partial_l)^{\dagger} \partial_m \partial_n \\
        = &
        \delta_{i m} \delta_{j n} \\
        - & \delta_{i j} (\partial_l \partial_l)^{\dagger} \partial_m \partial_n \\
        - & \delta_{i j} (\partial_k \partial_k)^{\dagger} \partial_m \partial_n \\
        + & \delta_{i j} (\partial_k \partial_k)^{\dagger}
        \underbrace{\partial_\alpha \partial_\alpha (\partial_l
        \partial_l)^{\dagger}}_{1}
        \partial_m \partial_n \\
        = & \delta_{i m} \delta_{j n} + (-2 + 1) \delta_{i j} (\partial_k \partial_k)^{\dagger} \partial_m \partial_n \\
        = & \pi_{i j m n}.
    \end{split}
    \end{equation}
    Since $\pi \pi = \pi$, we
    can conclude that $\pi$ is idempotent.
\end{proof}

The vector-projector $\pi_{i j}$ and tensor-projector $\pi_{i j \alpha \beta}$
satisfy the following commutation property for the tensor-divergence.

\begin{theorem} \label{th:projection-commutation}
    Projection and tensor-divergence commute, i.e.\ for all stress tensors
    $\sigma \in U^{3 \times 3}$, we have
    \begin{equation}
        \pi_{i j} \partial_k \sigma_{j k} = \partial_j \pi_{i j \alpha \beta} \sigma_{\alpha \beta}.
    \end{equation}
\end{theorem}

\begin{proof}
    Let $\sigma \in U^{3 \times 3}$ be a stress tensor.
    The tensor-divergence $\partial_j$ of the projected stress tensor $\pi_{i j \alpha \beta} \sigma_{\alpha \beta}$ is
    \begin{equation}
    \begin{split}
        \partial_j \pi_{i j \alpha \beta} \sigma_{\alpha \beta}
        & = \partial_j \left( \delta_{i \alpha} \delta_{j \beta}
            - \delta_{i j} (\partial_k \partial_k)^{\dagger} \partial_\alpha
            \partial_\beta \right) \sigma_{\alpha \beta} \\
        & = \left(
            \delta_{i \alpha} \partial_\beta -
            \partial_i (\partial_k \partial_k)^{\dagger} \partial_\alpha \partial_\beta \right)
            \sigma_{\alpha \beta} \\
        & = \left(
            \delta_{i \alpha} -
            \partial_i (\partial_k \partial_k)^{\dagger} \partial_\alpha \right)
            \partial_\beta \sigma_{\alpha \beta} \\
        & = \pi_{i \alpha} \partial_\beta \sigma_{\alpha \beta},
    \end{split}
    \end{equation}
    which is the projected tensor-divergence of $\sigma$.
\end{proof}

Note that we use the same symbol $\pi$ for both the vector and stress tensor projectors.
It should be clear from the context which version is used.

\section{FVM of the incompressible Navier-Stokes equations in alternative notation}
\label{sec:cartesian-notation}

We here present the incompressible Navier-Stokes equations in an alternative
notation to highlight how the RST in the FVM equation becomes non-local and non-symmetric. 
In Cartesian notation, the incompressible Navier-Stokes equations read
\begin{align}
    \partial_x u_x + \partial_y u_y + \partial_z u_z & = 0, \\
    \partial_t u_x + \partial_x \left( \sigma_{x x} + p \right) + \partial_y \sigma_{x y} + \partial_z \sigma_{x z} & = 0, \\
    \partial_t u_y + \partial_x \sigma_{y x} + \partial_y \left( \sigma_{y y} + p \right) + \partial_z \sigma_{y z} & = 0, \\
    \partial_t u_z + \partial_x \sigma_{z x} + \partial_y \sigma_{z y} + \partial_z \left( \sigma_{z z} + p \right) & = 0,
\end{align}
where
\begin{align}
    \sigma
    & \coloneq u \otimes u - \nu \left( \nabla u + \nabla u^T \right), \\
    u \otimes u
    & \coloneq \begin{pmatrix}
        u_x u_x & u_x u_y & u_x u_z \\
        u_y u_x & u_y u_y & u_y u_z \\
        u_z u_x & u_z u_y & u_z u_z
    \end{pmatrix}, \\
    \nabla u
    & \coloneq \begin{pmatrix}
        \partial_x u_x & \partial_y u_x & \partial_z u_x \\
        \partial_x u_y & \partial_y u_y & \partial_z u_y \\
        \partial_x u_z & \partial_y u_z & \partial_z u_z
    \end{pmatrix}.
\end{align}
In vector notation, we can write the equations as
\begin{equation}
    \nabla \cdot u = 0, \quad \partial_t u + \nabla \cdot
    \left( \sigma + p \delta \right) = 0.
\end{equation}
If we eliminate the pressure, this system can be written in projection-form as
three equations
\begin{align}
    \partial_t u_x +
    \partial_x r_{x x} +
    \partial_y r_{x y} +
    \partial_z r_{x z}
    & = 0, \\
    \partial_t u_y +
    \partial_x r_{y x} +
    \partial_y r_{y y} +
    \partial_z r_{y z}
    & = 0, \\
    \partial_t u_z +
    \partial_x r_{z x} +
    \partial_y r_{z y} +
    \partial_z r_{z z}
    & = 0,
\end{align}
or, in vector notation,
\begin{equation}
    \partial_t u + \nabla \cdot r = 0,
\end{equation}
where
\begin{equation}
    r \coloneq \pi \sigma \coloneq \sigma + p \delta
\end{equation}
is the projected stress tensor,
\begin{equation}
    p \coloneq -\triangle^\dagger \nabla \cdot \nabla \cdot \sigma,
\end{equation}
is a pressure field that depends non-locally on $\sigma$,
\begin{equation}
    \triangle \coloneq \partial_{x x} + \partial_{y y} + \partial_{z z},
\end{equation}
is the Laplacian,
$A^\dagger : b \mapsto x$ gives the unique solution (up to a constant) to the equation $A x =
b$ (we can for example design $A^\dagger$ to return the solution with zero
mean),
and
\begin{equation}
\begin{split}
    \nabla \cdot \nabla \cdot \sigma
    & \coloneq
    \partial_{x x} \sigma_{x x}
    + \partial_{y y} \sigma_{y y}
    + \partial_{z z} \sigma_{z z} \\
    & + 2 \left( \partial_{x y} \sigma_{x y}
    + \partial_{x z} \sigma_{x z}
+ \partial_{y z} \sigma_{y z} \right).
\end{split}
\end{equation}

\subsection{The stress tensor in the FVM equation}

Here we will write the FVM equation in ``unresolved'' form, without ever
introducing the resolved part $\nabla^h \cdot r^h(\bar{u}^h)$ and the numerical
stress tensor $r^h$.
The goal is to understand the properties of the stress tensor.

The FVM equation is
\begin{equation}
    \partial_t \bar{u}^h + \overline{\nabla \cdot \pi \sigma}^h = 0.
\end{equation}
For the tensor divergence,
the discrete filter-swap property \eqref{eq:ns-filter-swap} can be written as
$\overline{\nabla \cdot r}^h = \nabla^h \cdot \bar{r}^{h, *}$ for all tensors $r$,
where
$\nabla^h \coloneq (\partial^h_x, \partial^h_y, \partial^h_z)$ and
\begin{equation}
    \bar{r}^{h, *} \coloneq \begin{pmatrix}
        \bar{r}_{x x}^{h, x} &
        \bar{r}_{x y}^{h, y} &
        \bar{r}_{x z}^{h, z} \\
        \bar{r}_{y x}^{h, x} &
        \bar{r}_{y y}^{h, y} &
        \bar{r}_{y z}^{h, z} \\
        \bar{r}_{z x}^{h, x} &
        \bar{r}_{z y}^{h, y} &
        \bar{r}_{z z}^{h, z}
    \end{pmatrix}
\end{equation}
contains the surface-averaging filters 
$f^{h, x} \coloneq g^h_y g^h_z$,
$f^{h, y} \coloneq g^h_x g^h_z$, and
$f^{h, z} \coloneq g^h_x g^h_y$
for 1D top-hat filters $g^h_i$ in direction $i$.
This gives the $\nabla^h$-conservation law for $\bar{u}^h$:
\begin{equation}
    \partial_t \bar{u}^h + \nabla^h \cdot \overline{\pi \sigma}^{h, *} = 0.
\end{equation}
The stress tensor appearing the the $\nabla^h$-conservation law for $\bar{u}^h$ is therefore
\begin{equation}
\begin{split}
    \overline{\pi \sigma}^{h, *}
    & = \overline{\sigma + p \delta}^{h, *} \\
    & = \begin{pmatrix}
        \bar{\sigma}_{x x}^{h, x} &
        \bar{\sigma}_{x y}^{h, y} &
        \bar{\sigma}_{x z}^{h, z} \\
        \bar{\sigma}_{y x}^{h, x} &
        \bar{\sigma}_{y y}^{h, y} &
        \bar{\sigma}_{y z}^{h, z} \\
        \bar{\sigma}_{z x}^{h, x} &
        \bar{\sigma}_{z y}^{h, y} &
        \bar{\sigma}_{z z}^{h, z}
    \end{pmatrix} +
    \begin{pmatrix}
        \bar{p}^{h, x} & 0 & 0 \\
        0 & \bar{p}^{h, y} & 0 \\
        0 & 0 & \bar{p}^{h, z}
    \end{pmatrix}.
\end{split}
\end{equation}
This tensor is clearly non-symmetric, since
\begin{equation}
    f^{h, x} \neq f^{h, y} \neq f^{h, z}.
\end{equation}
Furthermore, the surface-averaged pressure tensor
\begin{equation}
    \overline{p \delta}^{h, *} =
    \begin{pmatrix}
        \bar{p}^{h, x} & 0 & 0 \\
        0 & \bar{p}^{h, y} & 0 \\
        0 & 0 & \bar{p}^{h, z}
    \end{pmatrix}
\end{equation}
is not an isotropic tensor
(it cannot be written as $q^h \delta$ for some scalar field $q^h$).
This is one of the reasons why the volume-averaged field $\bar{u}^h$ is not
$\nabla^h$-divergence-free, i.e.\ $\nabla^h \cdot \bar{u}^h \neq 0$.

Define the $\nabla^h$-divergence-free part of $\bar{u}^h$ as
\begin{equation}
    \bar{u}^{h, \pi} \coloneq \pi^h \bar{u}^h \coloneq
    \bar{u}^h - \nabla^h \left( \triangle^h \right)^\dagger \nabla^h \cdot \bar{u}^h,
\end{equation}
where $\pi^h$
is a vector-version of a discrete pressure projector 
and $\triangle^h \coloneq \partial^h_{x x} + \partial^h_{y y} + \partial^h_{z z}$.
The equation for $\bar{u}^{h, \pi}$ is
\begin{equation}
    \partial_t \bar{u}^{h, \pi} +
    \nabla^h \cdot \pi^h \overline{\pi \sigma}^{h, *} = 0,
\end{equation}
where
$\pi^h : \sigma \mapsto \sigma -
( \triangle^h )^\dagger \nabla^h \cdot \nabla^h \cdot \sigma$
here denotes the tensor version of the discrete pressure projector
(this should be clear from context).
We used \cref{th:projection-commutation} to swap $\nabla^h$ and $\pi^h$.
In the tensor $\pi^h \overline{\pi \sigma}^{h, *}$,
all the isotropic parts of $\overline{\pi \sigma}^{h, *}$ are absorbed by $\pi^h$.
But since the tensor $\overline{p \delta}^{h, *}$ is non-isotropic,
it does not get absorbed, and cannot be removed from the expression
$\pi^h \overline{\pi \sigma}^{h, *}$.
This is also why the pressure projector $\pi$ cannot be
taken outside the surface-averaging filter.
The stress tensor in the $\nabla^h$-conservation law for $\bar{u}^{h, \pi}$ is therefore
non-local in the external unresolved field $u$, which $p$ depends on.
This also remains the case even if we \emph{reintroduce} the incompressibility
constraint for $\bar{u}^{h, \pi}$ as
\begin{equation}
    \nabla^h \cdot \bar{u}^{h, \pi} = 0, \quad
    \partial_t \bar{u}^{h, \pi} + \nabla^h \cdot
    \left( \overline{\pi \sigma}^{h, *} + q^h \delta \right) = 0,
\end{equation}
where $q^h$ is the unique pressure (up to a constant) such that
$\overline{\pi \sigma}^{h, *} + q^h \delta = \pi^h \overline{\pi \sigma}^{h, *}$
and thus $\bar{u}^{\pi, h}$ remains $\nabla^h$-divergence-free.
Importantly, $q^h \neq \bar{p}^h$ is not the filtered pressure field,
since $\nabla^h$-incompressibility and $\nabla$-incompressibility are different
types of constraints.
The original pressure field $p$ is therefore still present in the stress
tensor, even though the equations are in incompressibility-constraint-form.

The fact that the unresolved stress tensor in the equation for
$\bar{u}^{h, \pi}$ is non-local in $u$
also suggest that a closure model for this unresolved stress tensor should be
non-local in $\bar{u}^{h, \pi}$.
If $m(\bar{u}^{h, \pi}) \approx \overline{\pi \sigma(u)}^{h, *}$
is a closure model for the total unresolved stress tensor,
then it should ideally be
\begin{enumerate}
    \item non-symmetric ($m^T \neq m$),
    \item non-local ($m(\bar{u}^{h, \pi})(x)$
        should depend on $\bar{u}^{h, \pi}(x + d)$
        for potentially large displacements $d \in \mathbb{R}^3$).
\end{enumerate}

\end{document}